\documentclass[12pt,reqno]{amsart}
\usepackage{amssymb,amsmath,amsthm,secdot,bbm,mathrsfs,mathtools,thmtools}

\usepackage[margin=1.2in]{geometry}

\usepackage[normalem]{ulem}
\usepackage[
bookmarks=true,
bookmarksnumbered=true,
colorlinks=true, pdfstartview=FitV, linkcolor=blue, citecolor=blue,
urlcolor=blue]{hyperref}

\usepackage[shortlabels]{enumitem}
\usepackage{color}

\usepackage{microtype}
\usepackage[nameinlink,capitalize]{cleveref}
\hypersetup{%
    bookmarksnumbered, bookmarksopen=true, bookmarksopenlevel=1,%
}

\DeclareMathOperator{\1}{\mathbbm{1}}
\newcommand{\eqlaw}{\overset{\mathrm{law}}{=}}
\def\eps{\varepsilon}
\def\phi{\varphi}
\newcommand{\dd}{\mathrm{d}}
\renewcommand{\ge}{\geqslant}
\renewcommand{\le}{\leqslant}

\newcommand{\R}{\mathbb{R}}

\newtheorem{theorem}{Theorem}[section]
\newtheorem{lemma}[theorem]{Lemma}
\newtheorem{proposition}[theorem]{Proposition}
\theoremstyle{definition}\newtheorem{remark}[theorem]{Remark}
\theoremstyle{definition}\newtheorem{definition}[theorem]{Definition}









\crefname{Lemma}{Lemma}{Lemmas}
\crefname{Theorem}{Theorem}{Theorems}

\title{From skew to reflected SPDEs}

\author{Cyril Labb\'e}
\address{Université Paris Cité and Sorbonne Université, CNRS, Laboratoire de Probabilités, Statistique et Modélisation, F-75013 Paris, France and Institut Universitaire de France (IUF).}
\email{clabbe@lpsm.paris}

\author{Thomas Le Guerch}
\address{Université Paris Cité and Sorbonne Université, CNRS, Laboratoire de Probabilités, Statistique et Modélisation, F-75013 Paris, France.}
\email{leguerch@lpsm.paris}

\author{Lorenzo Zambotti}
\address{Sorbonne Université and Université Paris Cité, CNRS, Laboratoire de Probabilit\'es, Statistique et Mod\'elisation, UMR 8001, F-75205 Paris, France and Institut Universitaire de France (IUF).}
\email{zambotti@lpsm.paris}

\begin{document}

\maketitle

\begin{abstract}
In this article, we prove that the solution of the skew stochastic heat equation converges to the solution of the stochastic heat equation with reflection in the limit of infinite skewness. The result holds at equilibrium. This provides an SPDE analog of the link between skew and reflected Brownian motion. One of the intermediate results is the convergence in law of a penalised Brownian bridge to a 3-Bessel bridge.
\end{abstract}



\setcounter{tocdepth}{1}
\tableofcontents

\section{Introduction}

\subsection{The skew and the reflected Brownian motion} Consider the following stochastic differential equation
\begin{equation}\label{eq:skew_BM}
	X_t=X_0+B_t+\beta L_t^0(X)\,,\quad t\ge0\,,
\end{equation}
where $B$ is a Brownian motion and $L_t^0(X)$ is the symmetric local time at $0$ of $X$, i.e.
\begin{equation}\label{eq:local_time}
	L_t^0(X):=\lim_{\eps\downarrow0}\frac{1}{2\eps}\int_0^t\1_{[-\eps,+\eps]}(X_s)\,\dd s\,.
\end{equation}
It was proved in \cite{Harrison1981} that (\ref{eq:skew_BM}-\ref{eq:local_time}) admits a unique solution if $|\beta|\le1$ and no solution at all if $|\beta|>1$.
Letting $\alpha=(\beta+1)/2$, $X$ is equal in law to a
Brownian motion whose excursions are chosen to be positive, respectively negative,
independently of each other, and each with probability $\alpha$, resp. $1-\alpha$. See \cite{Lejay06} for a nice review of the skew BM.

Consider now the following stochastic differential equation with reflection
\begin{equation}\label{eq:reflected_BM}
	\left\{\begin{aligned}
		X_t=X_0+B_t+L_t\,,\quad X\ge0\,,\\
		\dd L\ge0\,,\quad\int_0^\infty X_t\,\dd L_t=0
	\end{aligned}\right.
\end{equation}
where $B$ is a Brownian motion and the unknown is the pair $(X,L)$, see \cite[VI.5]{Revuz1999}. 
It turns out that the reflecting process $L$ is equal to the symmetric local time of $X$ defined as in \eqref{eq:local_time}, namely $L_t=L_t^0(X)$. Indeed, the skew Brownian motion with $\alpha=\beta=1$ is equal to the reflected Brownian motion, and therefore
\eqref{eq:reflected_BM} is a special case of (\ref{eq:skew_BM}-\ref{eq:local_time}).

In \cite{Harrison1981} it was also proved that equations (\ref{eq:skew_BM}-\ref{eq:local_time}) (resp. \eqref{eq:reflected_BM}) arise as the scaling limit of
the random walk with an asymmetry at $0$ (resp. the random walk with reflection).

\subsection{A Skew SPDE} In infinite dimension, the picture is quite different. A possible SPDE counterpart of (\ref{eq:skew_BM}-\ref{eq:local_time}) is given by 

\begin{equation}\label{eq:SPDE_skew}
	\left\{
	\begin{aligned}
		&\partial_t u_t(x)=\frac12 \partial^2_{xx}u_t(x) +\kappa\delta_0(u_t(x))+\dot{W}_{t}(x),\quad x\in[0,1],\\
		&u_t(0)=u_t(1)=0\,,\quad t\ge0\,,
	\end{aligned}
	\right.
\end{equation}
where $\dot W$ is a Gaussian space-time white noise, $\kappa\in\R$ and $\delta_0$ is a Dirac mass at $0$. This equation was first introduced in \cite{Bounebache2012} and called the skew stochastic heat equation.
The authors constructed a Markov process via Dirichlet form techniques and showed that it satisfies a stochastic equation involving local times of the unknown.
A formulation relying on local times allows one to avoid the $\kappa\delta_0(u)$ term which is ill-defined.
They further showed the convergence of two approximation procedures to the equation. 
More recently, \cite{Athreya2023,Athreya2025} introduced a formulation of \eqref{eq:SPDE_skew} (with $x\in\R$ or with periodic boundary conditions) that avoids local times, and a new approach to such SPDEs based on the
Stochastic Sewing lemma \cite{LeSSL}, obtaining pathwise uniqueness for a class of equations that includes \eqref{eq:SPDE_skew}.
The case of Dirichlet boundary conditions was recently
considered in \cite{Butkovsky2026}, where uniqueness {\it in law} and ergodicity for a class of equations including \eqref{eq:SPDE_skew}
was proved using Stochastic Sewing.

Similarly to \eqref{eq:reflected_BM}, \cite{Nualart1992} defines the stochastic heat equation with reflection as the equation on the pair $(u^+,\eta)$

\begin{equation}\label{eq:SPDE_NP}
	\left\{
	\begin{aligned}
		&\partial_t u^+_t(x)=\frac12 \partial^2_{xx}u^+_t(x) +\eta_t(x)+\dot{W}_{t}(x),\quad t\ge0,\, x\in[0,1],\\
		&u^+_t(0)=u^+_t(1)=0\,,\quad t\ge0\,,\\
		&u^+_0\ge0\,,\quad\int_{[0,\infty)\times[0,1]} u^+\,\dd\eta=0,
	\end{aligned}
	\right.
\end{equation}
where $\eta$ is a random measure on $[0,\infty)\times[0,1]$, $u^+$ is a.s. continuous on $[0,\infty)\times[0,1]$ and $\dot{W}$ is space-time Gaussian white noise.

\subsection{Main results} Our results are threefold. We first identify a reversible probability measure of the skew SPDE \eqref{eq:SPDE_skew}.

\begin{theorem}\label{theorem:invariant_measure}
For every $\kappa\in\mathbb R$, the unique solution of \eqref{eq:SPDE_skew} admits

\begin{equation}\label{eq:invariant_measure_ABLM}
	\mathbb Q^{\kappa}_{0,0}(\dd w)=Z_{\kappa}^{-1}\exp\left(-2\kappa\int_0^1\1_{\mathbb R_-}(w_x)\,\dd x\right)\mathbb W_{0,0}(\dd w)
\end{equation}
as a reversible probability measure, where $\mathbb W_{0,0}$ is the law of the Brownian bridge and $0<Z_{\kappa}<\infty$ is a normalisation constant.
\end{theorem}

Uniqueness of the invariant probability measure for the solution of the skew equation \eqref{eq:SPDE_skew} was proved in \cite{Butkovsky2026},
but existence was not addressed therein. Our result thus completes the picture by identifying explicitly the invariant measure.\\

Secondly, we identify the limit of the reversible measure $\mathbb Q_{0,0}^{\kappa}$ when the skewness parameter $\kappa$ goes to $+\infty$.

\begin{theorem}\label{theorem:static_convergence}
	One has the following convergence of measures
	\begin{equation}\label{eq:convergence_stationary_meas}
		\mathbb Q_{0,0}^{\kappa}\underset{\kappa\rightarrow+\infty}{\Longrightarrow}\mathbb P_{0,0}^{3}\,,
	\end{equation}
	where $\mathbb P_{0,0}^{3}$ is the law of the 3-Bessel bridge (and of the normalised Brownian excursion).
\end{theorem}

It was proved in \cite{Zambotti2001} that $\mathbb P_{0,0}^{3}$ is the only invariant probability measure for the solution of the reflected equation \eqref{eq:SPDE_NP}. Thus, a natural question is whether the convergence \eqref{eq:convergence_stationary_meas} also holds at the dynamical level (recall
that in general convergence of a sequence of stationary processes is strictly stronger than convergence of the associated invariant measures).
In other words, it is a natural question whether \eqref{eq:SPDE_skew} converges to \eqref{eq:SPDE_NP} as the skewness parameter $\kappa$ goes to $+\infty$, at least in equilibrium,
namely if the initial condition is distributed according to the invariant measure. 

Our third and main result provides a positive answer to this question.
\begin{theorem}\label{theorem:dynamical_convergence}
	Let $u^{\kappa}$ be the solution of \eqref{eq:SPDE_skew} started from $u_0^{\kappa}\eqlaw\mathbb Q_{0,0}^{\kappa}$ and let $u^+$ be the solution of \eqref{eq:SPDE_NP} started from $u^+_0\eqlaw\mathbb P_{0,0}^{3}$. One has the following convergence in law
	\begin{equation*}
		u^{\kappa}\underset{\kappa\rightarrow+\infty}{\Longrightarrow}u^+
	\end{equation*}
	in the space $\mathcal C([0,\infty)\times[0,1])$ of spacetime continuous functions endowed with the topology of uniform convergence over compact sets.
\end{theorem}

In order to appreciate this result, note that the drifts of both equations \eqref{eq:SPDE_skew} and \eqref{eq:SPDE_NP} are highly singular, and it is far from clear that
the former should converge to the latter, in any sense (see also the discussion at the end of the next subsection). The drift of \eqref{eq:SPDE_skew}, a Dirac mass at $0$, 
acts only when $u=0$, but by continuity of the solution its effect is to penalise the paths that become too negative. When the skewness $\kappa$ tends to $+\infty$, 
the stationary solution is forced to become non-negative by~\cref{theorem:static_convergence}. This shows why the knowledge of the invariant measure 
\eqref{eq:invariant_measure_ABLM} plays a crucial role, while a pathwise approach is at present not effective.

\subsection{Skew BM versus Skew heat equation}
	The invariant measure of the skew Brownian motion is given by $\alpha\1_{\R_+}(x)+(1-\alpha)\1_{\R_-}(x)\,\dd x$ (a non-normalisable measure). It is equivalent to the law of the Lebesgue measure $\dd x$ on $\R$, which is the invariant measure of the standard Brownian motion; the Radon-Nikodym derivative is noncontinuous and non-convex.

	The invariant (probability) measure of the skew SPDE \eqref{eq:SPDE_skew} is equivalent to the invariant (probability) measure of the stochastic heat equation. However, the Radon-Nikodym derivative is much smoother than for the skew BM: it admits a logarithmic derivative, see \cite{Bounebache2012}.

Unlike for \eqref{eq:skew_BM}-\eqref{eq:local_time}, where $\beta=1$ corresponds to the reflected Brownian motion,
there exists no $\kappa\in\R$ such that the solution of \eqref{eq:SPDE_skew} remains nonnegative. Equation \eqref{eq:SPDE_skew} is well posed for any $\kappa\in\R$, and
by~\cref{theorem:dynamical_convergence} the reflected equation is recovered only in the limit $\kappa\to+\infty$ (at least we can prove this in equilibrium).

	This is related to the fact that the skew and the reflected SDE are critical, in the sense that they are invariant in law under the Brownian scaling: $(\lambda^{-1/2}X_{\lambda t})_{t\ge 0}$ has the same law as $X$ with initial condition $\lambda^{-1/2}X_0$. The situation for the SPDEs is different. The reflected SPDE \eqref{eq:SPDE_NP} is critical,
	 in the sense that it is invariant under the same scaling as the linear stochastic heat equation: for $\lambda$ positive, $(\lambda^{-1/2}u^+_{\lambda^2 t}(\lambda x))_{t,x}$
	 satisfies the same equation as $u^+$. 
	On the other hand, the skew SPDE \eqref{eq:SPDE_skew} is subcritical, in the sense that, under the same scaling, the non-linearity is multiplied by $\lambda^{3/2}$ and
	therefore formally vanishes in the limit $\lambda\to 0$. 
	In infinite dimension, a subcritical skew equation becomes critical in the limit of infinite skewness.
	
	We point out another difference between the two settings: as already mentioned,
	the reflecting process $L$ of the reflected BM \eqref{eq:reflected_BM} is equal to the symmetric local time of $X$ defined in \eqref{eq:local_time}, 
	which allows one to write
	\[
	L_t=L^0_t(X) = \int_0^t \delta_0(X_s)\,\dd s, \qquad t\geq 0.
	\]
	The reflecting measure $\eta$ that appears in \eqref{eq:SPDE_NP} should have a similar representation as a positive continuous additive functional (PCAF) of the Markov process $(u^+_t)_{t\geq 0}$, namely:
	\begin{equation}\label{eq:representation_eta}
	\eta([0,t]\times A) = \int_0^t f_A(u^+_s)\,\dd s, \qquad t\geq 0,
	\end{equation}
	for a $f_A$ that could be a measure or a distribution on $C([0,1])$. This has been also investigated in \cite{Z02,Z04a}, where $\eta$ has been
	studied as an additive functional (in time) by identifying its Revuz measure, but an exact form for $f_A$
	remains elusive. This also explains why the convergence result of~\cref{theorem:dynamical_convergence} is far from trivial: the (singular) drift of \eqref{eq:SPDE_skew} does not converge to the (singular) drift of \eqref{eq:SPDE_NP} in any reasonable sense; actually, $\kappa\delta_0$ simply diverges as $\kappa\to+\infty$. 
	
	This also explains why Stochastic Sewing is not effective for the proof of the convergence in~\cref{theorem:dynamical_convergence}:
	in  \cite{Athreya2023,Athreya2025} a representation of the drift as $\int_0^t b(u_s)\,\dd s$ with $b\in {\mathcal C}^\gamma$ and $\gamma<0$ is
	crucial for the Stochastic Sewing Lemma to be applicable.

\subsection{Proof strategy}

Our strategy for the main result~\cref{theorem:dynamical_convergence} is to prove tightness of $(u^\kappa)_{\kappa}$ and identify the limit. Tightness of the initial conditions follows from the convergence~\eqref{eq:convergence_stationary_meas}. Uniform control of time increments follows from the Lyons-Zheng decomposition \cite{LyonsZheng}, which is available because $u^\kappa_0$ is distributed according to the invariant reversible measure of the equation. 

Unlike the recent papers on SPDEs with distributional drifts like \cite{Athreya2023,Athreya2025}, the proof of~\cref{theorem:dynamical_convergence} does not rely on Stochastic Sewing, since we do not have a representation like
\eqref{eq:representation_eta} for $\eta$ as an additive functional of $u^+$. On the other hand, pathwise uniqueness for the limit equation \eqref{eq:SPDE_NP} is very well 
understood  \cite{Nualart1992}. To prove convergence of $u^\kappa$ to $u^+$, we introduce $\eta^\kappa(\dd t,\dd x):=\kappa\delta_0(u^\kappa_t(x))\,\dd t\,\dd x$ and
show that $(u^\kappa,\eta^\kappa)$ (in equilibrium) converges to a solution of \eqref{eq:SPDE_NP}.

The convergence result of the invariant measures (\cref{theorem:static_convergence}) is fundamental in the proof of~\cref{theorem:dynamical_convergence} for 
two reasons. First, for tightness of $(u^\kappa)_\kappa$; secondly, and more importantly, for the non-negativity of the limit of $u^\kappa$ as $\kappa\to+\infty$. 
This crucial ingredient is necessary for the identification of this limit as a solution of \eqref{eq:SPDE_NP}. We note that the proof of~\cref{theorem:invariant_measure}, where the
invariant measure of \eqref{eq:SPDE_skew} is identified, relies on 
an approximation result from~\cite{Butkovsky2026}, which is itself based on Stochastic Sewing.

\subsection{Possible extensions}

	It would be natural to establish \cref{theorem:dynamical_convergence} out of equilibrium. In the present work, equilibrium is used to obtain that $(u^\kappa)_{\kappa>1}$ is tight and that any limit point as $\kappa\to+\infty$ is nonnegative: it is unclear how to deal with generic initial conditions. Note however that the identification of the limit presented in~\cref{Subsec:IdLimit} relies on pathwise arguments.
	
	One could explore the setting of unbounded spatial domains. To ensure the existence of an invariant probability measure, one can work on the entire line $\R$ with a damping term $-f(u^\kappa)$ or on the half line $\R_+$ with pinning at the origin, namely $u^\kappa(t,0)=0$. Let us mention that the latter setting was considered in~\cite{FauLab} at the level of the reflected SPDE \eqref{eq:SPDE_NP}: therein, the convergence of discrete models with reflection towards the stochastic heat equation with reflection on $\R_+$ with pinning was established.
	
	One could also aim for generalizing \cref{theorem:dynamical_convergence} by considering, instead of $\kappa \delta_0$, a family of measures $b_\kappa$ on $\R$ that ``converges'' to $+\infty \delta_0$. The main issue is to prove
	convergence of the stationary distributions, namely the analog of~\cref{theorem:static_convergence}, by replacing $\kappa\1_{\mathbb R_-}$ with an antiderivative of
	$-b_\kappa$ in \eqref{eq:invariant_measure_ABLM}. For the proof of~\cref{theorem:static_convergence} we rely on explicit formulae for densities of functionals of Brownian motion, which would need to be 
	extended in a more general setting (see the proof of~\cref{lemma:explicit_computations_Phi} in subsection \ref{Subsec:EstimatesDensity}).

\subsection{Outline of the paper}

\cref{section:stationary_measures} is dedicated to the proof of~\cref{theorem:static_convergence}.~\cref{section:SPDEs} provides rigorous definitions of the SPDEs under study and contains the proof of~\cref{theorem:invariant_measure}. \cref{Sec:DynCV} presents the proof of~\cref{theorem:dynamical_convergence}.

\section{Convergence of the stationary measures}\label{section:stationary_measures}

The goal of this section is to prove \cref{theorem:static_convergence}. We first introduce some notation. We denote by $w$ the canonical process on $\mathcal C([0,1],\R)$ and by $(\mathcal F_t)_{t\in[0,1]}$ the canonical filtration. We will use the following probability measures on $\mathcal C([0,1],\R)$:
\begin{itemize}
	\item $\mathbb W_{x}$ is the law of a Brownian motion starting at $x$
	\item $\mathbb W_{x,y}$ is the law of a Brownian bridge from $x$ to $y$
	\item $\mathbb P^{3}_x$ is the law of a $3$-dimensional Bessel process starting at $x$
	\item $\mathbb P^{3}_{x,y}$ is the law of a $3$-dimensional Bessel bridge from $x$ to $y$.
\end{itemize}
We will also consider, for any $\kappa > 0$ and $x\in\R$, the probability measure
\begin{equation}\label{eq:invariant_measure_ndirac}
	\mathbb Q^{\kappa}_{x,x}(\dd w)=(Z_{x,x}^{\kappa})^{-1}\exp\left(-2\kappa\int_0^1\1_{\R_-}(w_u)\,\dd u\right)\mathbb W_{x,x}(\dd w)\,,
\end{equation}
where $Z_{x,x}^{\kappa}\in(0,\infty)$ is a normalisation constant which ensures that $\mathbb Q^{\kappa}_{x,x}$ is a probability measure. Note that the notation is consistent with~\eqref{eq:invariant_measure_ABLM}.
Let us also emphasise that, in this section, the time parameter $t\in[0,1]$ is different from the time parameter $t>0$ appearing in the SPDEs \eqref{eq:SPDE_skew} and \eqref{eq:SPDE_NP}, where it rather corresponds to $x\in[0,1]$.

In the proofs presented in this section, the letter $C$ will denote a constant whose value is independent of all parameters at stake but whose value may change from line to line.

\subsection{Key propositions and proof of \cref{theorem:static_convergence}}

A key ingredient in our proof is the solution $X^\eps$ of the following SDE
\begin{equation}\label{eq:EDS_eps}
	\dd X_t^\eps=f_\eps(X_t^\eps)\,\dd t+\dd B_t\,\quad X_0^\eps=x\in\R\,,
\end{equation}
where $B$ is a standard Brownian motion and $f_\eps:\R\to\R_+$ is defined by
\begin{equation}\label{eq:feps}
f_\eps(x) := \frac1{\eps + x^+}, \qquad x\in\R, \quad \eps > 0.
\end{equation} 
We denote by $\mathbb P_x^{f_\eps}$ the law of $X^\eps$, and by $\mathbb P_{x,y}^{f_\eps}$ the law of the bridge of this diffusion from $(0,x)$ to $(1,y)$, for any $y\in\R$.
The key observation on which our proof relies relates the law of the bridge of this diffusion with the probability measure introduced in~\eqref{eq:invariant_measure_ndirac}:

\begin{proposition}\label{lemma:equality_measures}
		Fix $\kappa > 0$ and $x\in \R$, and set $\eps=(4\kappa)^{-1/2}$. Then
		\begin{equation*}
			\mathbb Q_{x,x}^{\kappa}=\mathbb P_{x,x}^{f_\eps}\,.
		\end{equation*}
\end{proposition}

\cref{lemma:equality_measures} shows that the measure $\mathbb Q_{x,x}^{\kappa}$ is the law of an explicit diffusion process, and this will allow us to make use of Itô calculus to attack the proof of \eqref{eq:convergence_stationary_meas}.

The monotonicity in $\eps$ of the drift $f_\eps$ allows one to prove that $X^\eps$, starting from some $x\ge 0$, converges as $\eps\downarrow 0$ to the solution $X$ of
$$ \dd X_t=\frac1{X_t}\,\dd t+\dd B_t,\quad X_0=x\,.$$
The process $X$ is the $3$-dimensional Bessel process starting at $x$. In other words
\begin{equation}\label{Eq:CVBessel}
	\mathbb P_0^{f_\eps}\underset{\eps\downarrow0}{\Longrightarrow}\mathbb P^{3}_0\,.
\end{equation}
We will not provide the details on this convergence (but it follows from the ingredients provided later on, see the proof of~\cref{theorem:static_convergence}). 

Given the identity stated in \cref{lemma:equality_measures}, the statement of \cref{theorem:static_convergence} follows if we can prove the convergence stated in~\eqref{Eq:CVBessel} at the level of bridges from $0$ to $0$. To that end, let us recall a few facts on the $3$-dimensional Bessel process. It admits transition densities given by
\begin{equation}\label{Eq:TransDensBessel}
	\begin{split}
	p^{3-\mathrm{Bessel}}_t(x,y) &= \sqrt{\frac{2}{\pi t}} \,\frac{y}{x} \,e^{-\frac{x^2+y^2}{2t}}\,\sinh\left(\frac{xy}{t}\right)\,,\quad x,y > 0\,,\\
	p^{3-\mathrm{Bessel}}_t(0,y) &= \frac{\sqrt{2} y^2}{\sqrt{\pi}t^{3/2}}\,e^{-\frac{y^2}{2t}}\,,\quad y>0\,,
	\end{split}
\end{equation}
see~\cite[Chap XI.1, p.446]{Revuz1999}. Furthermore for any $x>0$ and any $t\in (0,1)$
\begin{equation}\label{Eq:p3ratio}
	\frac{p^{3-\mathrm{Bessel}}_{1-t}(x,0)}{p^{3-\mathrm{Bessel}}_1(0,0)} := \lim_{y\downarrow 0} \frac{p^{3-\mathrm{Bessel}}_{1-t}(x,y)}{p^{3-\mathrm{Bessel}}_1(0,y)} =  \frac1{(1-t)^{3/2}} e^{-\frac{x^2}{2(1-t)}}\,.
\end{equation}
With these quantities at hand, the marginals of the $3$-dimensional Bessel bridge admit explicit expressions: for any $n\ge 1$ and any $0 < t_1 < \cdots <t_n < 1$, the vector $(w_{t_1},\ldots, w_{t_n})$ under $\mathbb P^{3}_{0,0}$ admits the following density on $\R^n$
\begin{equation}\label{Eq:marginalsBesselBridge}
p^{3-\mathrm{Bessel}}_{t_1}(0,x_1)\dots p^{3-\mathrm{Bessel}}_{t_n-t_{n-1}}(x_{n-1},x_n)\,\frac{p^{3-\mathrm{Bessel}}_{1-t_n}(x_n,0)}{p^{3-\mathrm{Bessel}}_1(0,0)} \1_{\{x_1,\ldots,x_n > 0\}}\,,
\end{equation}
see~\cite[Chap XI.3, p.463]{Revuz1999}.

Our second ingredient is the convergence of the transition density of $X^\eps$ to that of the $3$-dimensional Bessel process as $\eps\downarrow0$:
\begin{proposition}\label{lemma:transition_densities_convergence}
	Fix $t>0$ and $\delta>0$.
	\begin{enumerate}
		\item \label{item:transition_densities1} Uniformly over all $x\ge \delta$ and $y> 0$
		\begin{equation*}
			p_t^{\eps}(x,y)\underset{\eps\downarrow0}{\longrightarrow}p^{3-\mathrm{Bessel}}_t(x,y)\,.
		\end{equation*}
		\item \label{item:transition_densities2} Uniformly over all $y\ge \delta$
		\begin{equation*}
			p_t^{\eps}(0,y)\underset{\eps\downarrow0}{\longrightarrow}p^{3-\mathrm{Bessel}}_t(0,y)\,.
		\end{equation*}
		\item \label{item:transition_densities3} Assume that $t<1$. Uniformly over all $x\ge \delta$, recalling \eqref{Eq:p3ratio}
		\begin{equation*}
			\frac{p_{1-t}^\eps(x,0)}{p_1^\eps(0,0)}\underset{\eps\downarrow0}{\longrightarrow}\frac{p^{3-\mathrm{Bessel}}_{1-t}(x,0)}{p^{3-\mathrm{Bessel}}_1(0,0)}\,.
		\end{equation*}
	\end{enumerate}
\end{proposition}
Our last proposition establishes a uniform control on the moments of a H\"older-norm of the bridges, thus implying tightness. One could prove tightness
directly without resorting to such a uniform control of moments, but this will be used in the proof of~\cref{prop:tightness}.
\begin{proposition}\label{lemma:tightness}
	For any $\eta \in (0,1/2)$ and any $p\ge 1$
	$$ \sup_{\eps \in (0,1]} \mathbb P_{0,0}^{f_\eps}[\|w\|_{\mathcal{C}^\eta}^p] < \infty\;.$$
\end{proposition}

With all these intermediate results at hand, we can complete the proof of~\cref{theorem:static_convergence}.
\begin{proof}[Proof of~\cref{theorem:static_convergence}]
	By~\cref{lemma:equality_measures} we know that $\mathbb P_{0,0}^{f_\eps}$ coincides with $\mathbb Q_{0,0}^{\kappa}$ provided $\kappa = \eps^{-2}/ 4$. We thus need to prove that $\mathbb P_{0,0}^{f_\eps}$ converges to $\mathbb P^3_{0,0}$ as $\eps \downarrow 0$. This is carried out below.\\
	\textbf{Convergence of the finite dimensional distributions.} Fix $n\ge1$, $t_1<\dots<t_n\in(0,1)$ and $a_1<b_1,\dots,a_n<b_n\in(0,\infty)$. Then
	\begin{equation*}
	\begin{split}
		&\mathbb P_{0,0}^{f_\eps}\left(w_{t_1}\in [a_1,b_1],\dots,w_{t_n}\in [a_n,b_n]\right)\\
		=&\int_{a_1}^{b_1}\dots\int_{a_n}^{b_n}p^{\eps}_{t_1}(0,x_1)\dots p^{\eps}_{t_n-t_{n-1}}(x_{n-1},x_n)\,\frac{p^{\eps}_{1-t_n}(x_n,0)}{p^{\eps}_1(0,0)}\, \dd x_1\dots\dd x_n\\
		\underset{\eps\downarrow0}{\longrightarrow}&\int_{a_1}^{b_1}\dots\int_{a_n}^{b_n}p^{3-\mathrm{Bessel}}_{t_1}(0,x_1)\dots p^{3-\mathrm{Bessel}}_{t_n-t_{n-1}}(x_{n-1},x_n)\,\frac{p^{3-\mathrm{Bessel}}_{1-t_n}(x_n,0)}{p^{3-\mathrm{Bessel}}_1(0,0)}\, \dd x_1\dots\dd x_n\\
		=&\mathbb P^3_{0,0}\left(w_{t_1}\in [a_1,b_1],\dots,w_{t_n}\in [a_n,b_n]\right)
	\end{split}
	\end{equation*}
	by the Dominated Convergence Theorem and~\cref{lemma:transition_densities_convergence}. This establishes the convergence of the marginals at times $0<t_1 < \cdots < t_n<1$ on any product of closed intervals of $(0,\infty)$. By a monotone class argument, this extends to any Borel set of $(0,\infty)^n$. Since for all $t\in (0,1)$, $\mathbb P^3_{0,0}$-a.s.~$w_t > 0$, the convergence remains true for any Borel set of $\R^n$. We have thus proven the convergence of the marginals at times $t_1 < \cdots <t_n$, provided $t_1 > 0$ and $t_n <1$. The extension to $t_1= 0$ and/or $t_n = 1$ is straightforward since under either $\mathbb P_{0,0}^{f_\eps}$ or $\mathbb P^3_{0,0}$, $w_0=w_1=0$ a.s.
	
	\smallskip
	
	\noindent \textbf{Tightness.} Tightness is a consequence of~\cref{lemma:tightness}. By the Markov inequality, for every $K>0$, for all $p\ge1$, for every $\eps\in(0,1]$,
	\[
		\mathbb P_{0,0}^{f_\eps}[\|w\|_{\mathcal C^q}> K]\le K^{-p}\,\mathbb P_{0,0}^{f_\eps}[\|w\|_{\mathcal C^q}^p]
	\]
	for some $q>0$ given by~\cref{lemma:tightness} such that $\sup_{\eps\in(0,1]}\mathbb P_{0,0}^{f_\eps}[\|w\|_{\mathcal C^q}^p]<\infty$. We conclude using the Arzelà-Ascoli theorem, taking $K$ large enough.
	\end{proof}

The remaining of this section is devoted to the proofs of \cref{lemma:equality_measures}, \cref{lemma:transition_densities_convergence} and~\cref{lemma:tightness}.

\subsection{Transition densities of the diffusion}

We start by identifying the Radon-Nikodym derivative of the law of the diffusion $X^\eps$ w.r.t.~the law of the Brownian motion. Let $F_\eps:\R\to\R$ be defined by
\[
F_\eps(x)=\ln\left(\frac{\eps+x}\eps\right)\1_{(x>0)}+\frac x\eps\1_{(x\le 0)},
\]
so that $F_\eps' = f_\eps$. Note that
\[
e^{F_\eps(x)}=\frac{\eps+x}\eps\1_{(x>0)}+e^{x/\eps}\,\1_{(x\le 0)}.
\]
\begin{lemma}\label{lemma:law_of_X_eps}
	Fix $\eps>0$. For any $t\in [0,1]$, the law of $X^{\eps}$ satisfies
	\begin{equation}\label{eq:law_of_X_eps}
		\left.\frac{\dd \mathbb P_x^{f_\eps}}{\dd \mathbb W_x}\right|_{\mathcal F_t}(w):=\exp\left(
		F_\eps(w_t)-F_\eps(x)-\frac{\eps^{-2}}{2}\int_0^t\1_{(w_s\le0)}\,\dd s\right)\,.
	\end{equation}
\end{lemma}
\begin{proof}
	Recall that under $\mathbb W_x$ the canonical process $w$ is a Brownian motion starting from $x$. We set $L_t=\int_0^tf_\eps(w_s)\,\dd w_s$.
	As $f_\eps$ is bounded by $\eps^{-1}$, one has that
	\begin{equation*}
		\mathbb E\left[\exp\left(\frac12\langle L\rangle_1\right)\right]=\mathbb E\left[\exp\left(\frac12\int_0^1f_\eps^2(w_s)\,\dd s\right)\right]\le\exp\left(\frac{\eps^{-2}}{2}\right)<\infty\,,
	\end{equation*}
	and by Novikov's criterion~\cite[Prop. VIII.1.15]{Revuz1999}, the exponential martingale $\mathcal E(L)_t:=\exp\left(L_t-\frac12\langle L\rangle_t\right)$, $t\in[0,1]$, is a true martingale starting from $1$ at time $0$. Consequently, by Girsanov's Theorem, the process $w'$ defined by 
	\begin{equation}\label{eq:SDE_girsanov}
		w'_t:=w_t-\int_0^tf_\eps(w_s)\,\dd s\,,\quad t\in[0,1]\,,
	\end{equation}
	is a Brownian motion under $\tilde{\mathbb W}_x(\dd w) := \mathcal E(L)_1(w) \, \mathbb W_x(\dd w)$. In particular under $\tilde{\mathbb W}_x$, $w$ has the same law as the solution $X^\eps$ of~\eqref{eq:EDS_eps}. To conclude the proof, we only need to prove that almost surely for any $t\in [0,1]$
	\begin{equation}\label{eq:RN_density_after_ito_trick}
		\mathcal E(L)_t(w)=\exp\left(
		F_\eps(w_t)-F_\eps(x)-\frac{\eps^{-2}}{2}\int_0^t\1_{(w_s\le0)}\,\dd s\right)\,.
	\end{equation}
	Recall that
	\begin{equation}\label{eq:RN_density_before_ito_trick}
		\mathcal E(L)_t(w)=\exp\left(\int_0^t f_\eps(w_s)\,\dd w_s-\frac12\int_0^t f_\eps^2(w_s)\,\dd s\right)\,.
	\end{equation}
	The function $f_\eps$ is $C^1$ except at $x=0$ where the derivative has a discontinuity:
	$$ f_\eps'(x) = -\frac{\1_{(x>0)}}{(\eps+x)^2}\;.$$
	Since $f_\eps'\le 0$, we deduce that $F_\eps$ is concave. By the It\^o-Tanaka formula~\cite[Th VI.1.5]{Revuz1999}, we deduce that almost surely for any $t\in [0,1]$
	\begin{equation}\label{eq:ito_trick2}
	F_\eps(w_t)-F_\eps(x) = \int_0^t f_\eps(w_s)\,\dd w_s + \frac12\int_0^t f_\eps'(w_s)\,\dd s\,.
	\end{equation}
	Notice that for every $x\neq0$,
	\begin{equation}\label{eq:RN_identity}
		f_\eps'(x)+f_\eps^2(x)=-\frac{\1_{(x>0)}}{(\eps+x^+)^2}+\frac{1}{(\eps+x^+)^2}=\frac{\1_{(x\le 0)}}{(\eps+x^+)^2}=\eps^{-2}\1_{(x\le 0)}\,.
	\end{equation}
	Plugging~\eqref{eq:ito_trick2} and~\eqref{eq:RN_identity} into~\eqref{eq:RN_density_before_ito_trick},
	we obtain~\eqref{eq:RN_density_after_ito_trick}.
\end{proof}

Next, we establish a tractable expression for the transition kernel of $X^{\eps}$.  To that end, we let $g_t(x)$, $x\in\R$, be the density of the $\mathcal N(0,t)$ law. We also set for any $\lambda>0$ and any $x,y \in \R$
\begin{equation}\label{Eq:Phi}
	\Phi_t(\lambda;x,y)=\mathbb W_{(0,x),(t,y)}\left[\exp\left(-\lambda\int_0^t\1_{\R_-}(w_s)\,\dd s\right)\right],
\end{equation}
where $\mathbb W_{(0,x),(t,y)}$ is the law of the Brownian bridge starting from $x$ at time $0$ and ending at $y$ at time $t$. Note that $\mathbb W_{(0,x),(1,y)}$ coincides with $\mathbb W_{x,y}$.

\begin{lemma}\label{lemma:transition_densities}
	Fix $\eps>0$ and $t\in (0,1]$. One has for any $x,y\in\R$,
	\begin{equation}\label{eq:transition_densities}
		p_t^{\eps}(x,y)=
		\frac{e^{F_\eps(y)}}{e^{F_\eps(x)}}\, g_t(y-x) \, \Phi_t(\eps^{-2}/2;x,y)\,.
	\end{equation}
\end{lemma}
\begin{proof}
	Using \cref{lemma:law_of_X_eps}, we deduce that for every bounded measurable $\phi:\R\to\R$
	\begin{align*}
		&\mathbb P_x^{f_\eps}\left[\phi(w_t)\right]\\
		=&\mathbb W_x\left[\exp\left(
		F(w_t)-F(x)-\frac{\eps^{-2}}{2}\int_0^t\1_{(w_s\le0)}\,\dd s\right)\phi(w_t)\right]\\
		=&\int_{\R}\, \dd y\,g_t(y-x)\mathbb W_x\left[\left.\exp\left(
		F(w_t)-F(x)-\frac{\eps^{-2}}{2}\int_0^t\1_{(w_s\le0)}\,\dd s\right)\phi(w_t)
		\,\right|\,w_t=y\right]\\
		=&\int_{\R}\, \dd y\,\phi(y) \,
		\frac{e^{F_\eps(y)}}{e^{F_\eps(x)}}\,g_t(y-x) \,\mathbb W_{(0,x),(t,y)}\left[\exp\left(-\frac{\eps^{-2}}{2}\int_0^t\1_{(w_s\le0)}\,\dd s\right)\right].
	\end{align*}
By the definition of $\Phi_t$ in~\eqref{Eq:Phi}, the result follows.
\end{proof}

\noindent We can now prove the equality between $\mathbb P_{x,x}^{f_\eps}$ and $\mathbb Q_{x,x}^{\kappa}$, provided $\eps=(4\kappa)^{-1/2}$.
\begin{proof}[Proof of~\cref{lemma:equality_measures}]
	The map $y\mapsto \mathbb P_{x,y}^{f_\eps}$ is the unique continuous map w.r.t.~the topology of weak convergence of probability measures that satisfies
	$$ \mathbb P_{x}^{f_\eps}(\dd w) = \int_{y\in\R} \mathbb P_{x,y}^{f_\eps}(\dd w) \, p_1^\eps(x,y)\, \dd y\;.$$
	On the other hand, using the result of~\cref{lemma:law_of_X_eps} and disintegrating the measure $\mathbb W_x$ w.r.t.~the marginal at time $1$, we find
	\begin{align*}
		\mathbb P_{x}^{f_\eps}(\dd w) &=\exp\left(
		F_\eps(w_1)-F_\eps(x)-\frac{\eps^{-2}}{2}\int_0^1\1_{(w_s\le0)}\,\dd s\right)\mathbb W_x(\dd w)\\
		&= \int_{y\in\R} \exp\left(
		F_\eps(w_1)-F_\eps(x)-\frac{\eps^{-2}}{2}\int_0^1\1_{(w_s\le0)}\,\dd s\right)\mathbb W_{x,y}(\dd w) \,g_1(x-y)  \,\dd y\\
		&= \int_{y\in\R}  \frac{\exp\left(-\frac{\eps^{-2}}{2}\int_0^1\1_{(w_s\le0)}\,\dd s\right)\mathbb W_{x,y}(\dd w)}{\Phi_1(\eps^{-2}/2;x,y)} \, p_1^\eps(x,y)\, \dd y\,,
	\end{align*}
where the last line comes from~\cref{lemma:transition_densities}. We deduce that for all $y\in\R$
$$ \mathbb P_{x,y}^{f_\eps}(\dd w) = \frac{\exp\left(-\frac{\eps^{-2}}{2}\int_0^1\1_{(w_s\le0)}\,\dd s\right)\mathbb W_{x,y}(\dd w)}{\Phi_1(\eps^{-2}/2;x,y)}\,.$$
Specializing to $y=x$, we obtain the desired result.
\end{proof}

To exploit the identity~\eqref{eq:transition_densities} on the density of the diffusion, we need to collect estimates on the function $\Phi_t$. This is the purpose of the next two lemmas, whose proofs are postponed to~\cref{Subsec:EstimatesDensity}. First, we consider the situation where $x,y \ge 0$ and obtain an expansion in powers of $\lambda$ (one should think of $\lambda$ large since, in fine, it will be taken equal to $\eps^{-2}/2$).
\begin{lemma}\label{lemma:explicit_computations_Phi}
	For all $\lambda > 0$, $t > 0$ and $x,y \ge 0$ such that $x+y>0$, we have
	\begin{equation}\label{eq:explicit_computations_Phi}
		g_t(y-x)\,\Phi_t(\lambda;x,y)=\sqrt{\frac{2}{\pi t}}\,e^{-\frac{x^2+y^2}{2t}}\sinh\left(\frac{xy}{t}\right) + \frac1{\sqrt{\lambda}} \frac{x+y}{t^{3/2}\sqrt{\pi}}\,e^{-\frac{(x+y)^2}{2t}}+\frac1{\lambda} R(\lambda,t,x+y)\,,
	\end{equation}
	where $R(\lambda,t,x+y)$ satisfies
	\begin{equation}\label{Eq:R}
		|R(\lambda,t,x+y)| \le K \frac1{t^{3/2}} e^{-\frac{(x+y)^2}{4t}} \,,
	\end{equation}
	for some constant $K>0$ independent of $x+y>0$, $t>0$ and $\lambda>0$.
	
	Furthermore, for all $\lambda > 0$ and $t>0$, we have
	\begin{equation}\label{Eq:x=y=0}
		g_t(0)\,\Phi_t(\lambda;0,0)=\frac{1-e^{-\lambda t}}{\lambda t^{3/2} \sqrt{2\pi}}\,.
	\end{equation}
\end{lemma}

Second, we consider the situation where either $x=0$ and $y<0$, or $x<0$ and $y=0$. As $\eps \downarrow 0$, i.e.~$\lambda\uparrow\infty$, the corresponding transition densities are expected to vanish. In the next lemma, we establish a bound on these transition densities.

\begin{lemma}\label{lemma:explicit_computations_Phi2}
	There exists $C>0$ such that for all $x,y<0$, all $\lambda > 0$ and all $t>0$
	$$ g_t(x)\,\Phi_t(\lambda;x,0) \le C \left( e^{-\sqrt{\lambda} \frac{|x|}{2}} + e^{-\lambda \frac{t}{4}}\right)\frac{e^{-\frac{x^2}{4t}}}{\lambda t^{3/2}}\;,$$
	and
	$$ g_t(y)\,\Phi_t(\lambda;0,y) \le C \left( e^{-\sqrt{\lambda} \frac{|y|}{2}} + e^{-\lambda \frac{t}{4}}\right)\frac{e^{-\frac{y^2}{4t}}}{\lambda t^{3/2}}\;.$$
\end{lemma}

With these ingredients at hand, we can proceed with the:

\begin{proof}[Proof of \cref{lemma:transition_densities_convergence}]
	By \cref{lemma:transition_densities} and \cref{lemma:explicit_computations_Phi}, for any $x,y\ge 0$ such that $x+y > 0$ we have
	\begin{equation*}
		p_t^{\eps}(x,y)=\frac{\eps+y}{\eps+x} \left(\sqrt{\frac{2}{\pi t}}e^{-\frac{x^2+y^2}{2t}}\sinh\left(\frac{xy}{t}\right) \!+\! \sqrt{2}\eps \frac{x+y}{t^{3/2}\sqrt{\pi}}e^{-\frac{(x+y)^2}{2t}} \!+ 2\eps^2 R(\frac1{2\eps^2},t,x+y)\right)\!.
	\end{equation*}
	\noindent We start by proving Item~\eqref{item:transition_densities1}. Take $x\ge\delta$ and $y>0$. Using \eqref{Eq:TransDensBessel}, we find
	\begin{align*}
		|p_t^{\eps}(x,y)-p_t^{3-\textrm{Bessel}}(x,y)|\le& \left|\frac{\eps+y}{\eps+x} - \frac{y}{x}\right| \sqrt{\frac{2}{\pi t}}\,e^{-\frac{x^2+y^2}{2t}}\sinh\left(\frac{xy}{t}\right)\\
		&+ \frac{\eps+y}{\eps+x} \left(\sqrt{2}\eps \frac{x+y}{t^{3/2}\sqrt{\pi}}\,e^{-\frac{(x+y)^2}{2t}}+ 2\eps^2 |R(\frac1{2\eps^2},t,x+y)|\right)\,.
	\end{align*}
	Given the bound~\eqref{Eq:R}, it is straightforward to check that the r.h.s.~goes to $0$ as $\eps\downarrow 0$ uniformly over all $x\ge \delta$ and $y> 0$.

	\noindent We now prove Item~\eqref{item:transition_densities2}. Take $y\ge \delta$. Using \eqref{Eq:TransDensBessel}, we find
	\begin{align*}
		|p_t^{\eps}(0,y)-p_t^{3-\textrm{Bessel}}(0,y)|\le& \eps \frac{y \sqrt{2}}{t^{3/2}\sqrt{\pi}}\,e^{-\frac{y^2}{2t}}+ 2\eps(\eps+y)  |R(\frac1{2\eps^2},t,y)|\,.
	\end{align*}
	The bound~\eqref{Eq:R} ensures that the r.h.s.~goes to $0$ as $\eps\downarrow 0$ uniformly over all $y\ge \delta$.
	
	\noindent We now prove Item~\eqref{item:transition_densities3}. By \cref{lemma:transition_densities} and \cref{lemma:explicit_computations_Phi}, for all $x>0$
	\begin{align*}
		\frac{p_{1-t}^{\eps}(x,0)}{p_1^{\eps}(0,0)} &=  \frac{x}{(\eps+x)(1-e^{-\frac1{2\eps^2}})} \frac1{(1-t)^{3/2}} e^{-\frac{x^2}{2(1-t)}} +\frac{\eps}{\eps+x} \frac{\sqrt{2\pi}}{1-e^{-\frac1{2\eps^2}}}  R(\frac1{2\eps^2},1-t,x)\,,
	\end{align*}
	and the r.h.s.~is bounded uniformly over $\eps \in (0,1]$ and $x\ge \delta$. Furthermore, using \eqref{Eq:p3ratio}, we find
	\begin{align*}
		\left|\frac{p_{1-t}^{\eps}(x,0)}{p_1^{\eps}(0,0)}-\frac{p_{1-t}^{3-\textrm{Bessel}}(x,0)}{p_1^{3-\textrm{Bessel}}(0,0)}\right|\le& \left|1 - \frac{x}{(\eps+x)(1-e^{-\frac1{2\eps^2}})}\right| \frac1{(1-t)^{3/2}} e^{-\frac{x^2}{2(1-t)}}\\
		&+\frac{\eps}{\eps+x} \frac{\sqrt{2\pi}}{1-e^{-\frac1{2\eps^2}}}  |R(\frac1{2\eps^2},1-t,x)|\,.
	\end{align*}
	The bound~\eqref{Eq:R} ensures that the r.h.s.~goes to $0$ as $\eps\downarrow 0$ uniformly over all $x\ge \delta$.
\end{proof}

\subsection{Moments bounds}

We start by controlling the moments of the increments of the unconditioned diffusion $X^\eps$.

\begin{lemma}\label{lemma:holder_tightness_unconditionned}
	For any $\eta\in(0,1/2)$ and any $p\ge 2$
	$$\sup_{\eps \in (0,1]} \mathbb P_{0}^{f_\eps}\left[\left(\sup_{0 \le s < t \le 1} \frac{|w_t - w_s|}{|t-s|^\eta}\right)^{p}\right] < \infty\;.$$
\end{lemma}
\begin{proof}
	We denote by $X^\eps$ the solution of
	$$ \dd X_t^\eps=f_\eps(X_t^\eps)\,\dd t+\dd B_t\,\quad X_0^\eps=0\,,$$
	where $B$ is a Brownian motion starting from $0$. It suffices to prove that for any $p\ge 2$, there exists $C_p>0$ such that for all $0 \le s \le t \le 1$ and all $\eps \in (0,1]$
	$$ \mathbb{E} [|X_t^\eps - X_s^\eps|^p] \le C_p |t-s|^{p/2}\;.$$
	From now on, we fix $0\le s \le t \le 1$ and $\eps \in (0,1]$ and we establish bounds that hold uniformly in these parameters.
	
	We argue differently according to the sign of $X_s^\eps$. Let $(Y_r,r\ge s)$ be the solution of
	$$ \dd Y_r = \frac1{Y_r}\, \dd r + \dd B_r\;,\quad r\ge s\;,\qquad Y_s := (X_s^\eps)^+\;.$$
	On the event $\{X_s^\eps \ge 0\}$, we have
	$$ B_t-B_s \le X_t^\eps - X_s^\eps \le Y_t - Y_s\;,$$
	and consequently
	$$ |X_t^\eps - X_s^\eps|^p \le |B_t-B_r|^p + |Y_t-Y_s|^p\;.$$
	Since the $3$d-Bessel process has the law of the Euclidean norm of a $3$d-Brownian motion, we deduce that there exists a constant $C>0$, independent of $s,t$ and $\eps$, such that
	$$ \mathbb{E}\left[|X_t^\eps - X_s^\eps|^p \1_{\{X_s^\eps \ge 0\}}\right] \le C \, |t-s|^{p/2}\;.$$
	We have thus established the desired bound in the case where $X_s^\eps \ge 0$.
	
	Next, we observe that, whatever the sign of $X_s^\eps$, we have
	$$ B_t - B_s \le X_t^\eps - X_s^\eps \le B_t - B_s + \frac{t-s}{\eps}\;,$$
	and therefore there exists a constant $C>0$, independent of $s,t$ and $\eps$, such that
	$$ \mathbb{E}[|X_t^\eps - X_s^\eps|^p] \le C \left(|t-s|^{p/2} + \left(\frac{|t-s|}{\eps}\right)^p\right)\;.$$
	If $|t-s| \le \eps^2$, this provides the desired bound.
	
	It remains to cover the case where $X_s^\eps < 0$ and $|t-s| > \eps^2$. We let $\tau:=\inf\{r\ge s: X_r^\eps = 0\}$. Then we write
	$$|X_t^\eps - X_s^\eps| = |X_t^\eps - X_{t\wedge \tau}^\eps| + |X_{t\wedge \tau}^\eps - X_s^\eps|\;.$$
	The strong Markov property and the arguments above show that
	$$\mathbb{E}[|X_t^\eps - X_{t\wedge \tau}^\eps|^p] \le C|t-s|^{p/2}\;.$$
	We thus concentrate on bounding
	$$ \mathbb{E}\left[|X_{t\wedge \tau}^\eps - X_s^\eps|^p \1_{\{X_s^\eps < 0\}}\right] = A_1 + A_2\;,$$
	where
	\[
		A_1:= \mathbb{E}\left[|X_{t\wedge \tau}^\eps - X_s^\eps|^p \1_{\{X_s^\eps < 0; \tau \le t\}}\right],\qquad
		A_2:= \mathbb{E}\left[|X_{t\wedge \tau}^\eps - X_s^\eps|^p \1_{\{X_s^\eps < 0; \tau > t\}}\right]\;.
	\]
	On the event $\{\tau \le t\}$, we have $|X_{t\wedge \tau}^\eps - X_s^\eps| = |X_s^\eps|$. Consequently
	$$ A_1 \le \mathbb{E}[|X_s^\eps|^p \1_{\{X_s^\eps < 0\}}] = \int_{-\infty}^0 |x|^p \, p_s^{\eps}(0,x)\,\dd x \,.$$
	Let us prove that
	\begin{equation}\label{eq:claimnow}
	\int_{-\infty}^0 |x|^p \, p_s^{\eps}(0,x)\,\dd x\le C \, \eps^p.
	\end{equation}
		By~\eqref{eq:transition_densities} and since $\Phi_t(\lambda,0,x) \le 1$, we obtain
	\begin{align*}
		\int_{-\infty}^0 |x|^p \, p_s^{\eps}(0,x)\,\dd x &= \int_{-\infty}^0 |x|^p \, e^{\frac{x}{\eps}} g_t(x) \Phi_t(\lambda;0,x)\,\dd x\\
		&\le \eps^p \int_{-\infty}^0 \left(\frac{|x|}{\eps}\right)^p \, e^{\frac{x}{\eps}} g_t(x) \,\dd x\\
		&\le \eps^p \left( \sup_{y > 0} y^p \, e^{-y}\right) \int_{-\infty}^0 g_t(x) \,\dd x\\
		&\le C \eps^p\;.
	\end{align*}
	
	We turn to $A_2$. First consider the event $\{-\frac{t-s}{2\eps} < X_s^\eps < 0\}$ and note that if $\tau > t$ then necessarily
	$$ B_t-B_s + \frac{t-s}{\eps} < - X_s^\eps\;,$$
	and therefore
	$$ B_t-B_s < -\frac{t-s}{2\eps}\;.$$
	Consequently, setting $\zeta:=\frac{\sqrt{t-s}}{2\eps}$
	\begin{align*}
		\mathbb{E}\left[|X_{t\wedge \tau}^\eps - X_s^\eps|^p \1_{\{-\frac{t-s}{2\eps} < X_s^\eps < 0,\,\tau > t\}}\right] &= \mathbb{E}\left[|X_{t}^\eps - X_s^\eps|^p \1_{\{-\frac{t-s}{2\eps} < X_s^\eps < 0,\,\tau > t\}}\right]\\
		&\le \mathbb{E}\left[\left|B_t-B_s + \frac{t-s}{\eps}\right|^p \1_{\{B_t-B_s < -\frac{t-s}{2\eps}\}}\right]\\
		& = |t-s|^{p/2} \, \mathbb{E}\left[\left|Z + 2\zeta\right|^p \1_{\{Z < -\zeta\}}\right],
	\end{align*}
	where $Z\sim\mathcal N(0,1)$.	Now it remains to estimate
	\[
	\begin{split}
	 \sup_{\zeta>\frac12}\mathbb{E}\left[\left|Z + 2\zeta\right|^p \1_{\{Z < -\zeta\}}\right] & \le C \sup_{\zeta>\frac12}\mathbb{E}\left[(|Z|^p + \zeta^p) \1_{\{Z < -\zeta\}}\right]
	\\ & \le C \sup_{\zeta>\frac12}\left(1+ \zeta^p \,  e^{-\frac{\zeta^2}2}\right)<+\infty.
	\end{split}
	\]
	Second, we consider the event $\{X_s^\eps \le -\frac{t-s}{2\eps}\}$ and write
	\begin{align*}
		\mathbb{E}\left[ \left|X_{t\wedge \tau}^\eps - X_s^\eps\right|^p \1_{\{X_s^\eps \le -\frac{t-s}{2\eps},\,\tau > t\}}\right] &= \mathbb{E}\left[ |X_{t}^\eps - X_s^\eps|^p \1_{\{X_s^\eps \le -\frac{t-s}{2\eps},\,\tau > t\}}\right]\\
		&=\mathbb{E}\left[ \left|B_t - B_s + \frac{t-s}{\eps}\right|^p \1_{\{X_s^\eps \le -\frac{t-s}{2\eps},\,\tau > t\}}\right]\\
		&\le \mathbb{E}\left[ \left|B_t - B_s + \frac{t-s}{\eps}\right|^p \1_{\{X_s^\eps \le -\frac{t-s}{2\eps}\}}\right]\\
		&\le C \, |t-s|^{p/2} + C \mathbb{E}\Big[\left(\frac{t-s}{\eps}\right)^p\1_{\{X_s^\eps \le -\frac{t-s}{2\eps}\}} \Big]\;,
	\end{align*}
	and we note that
	\begin{align*}
		\mathbb{E}\Big[\left(\frac{t-s}{\eps}\right)^p\1_{\{X_s^\eps \le -\frac{t-s}{2\eps}\}} \Big] \le C\mathbb{E}\Big[ |X_s^\eps|^p\1_{\{X_s^\eps \le -\frac{t-s}{2\eps}\}} \Big] \le C \int_{-\infty}^0 |x|^p \, p_s^{\eps}(0,x) \, \dd x\;.
	\end{align*}
	By \eqref{eq:claimnow} and $|t-s|>\eps^2$ the last term is bounded by a term of order $|t-s|^{p/2}$.
\end{proof}
	
\begin{proof}[Proof of~\cref{lemma:tightness}]
	Take $t_0 := 2/3$. If we show that for any $\eta \in (0,1/2)$ and any $p\ge 1$
	$$ \sup_{\eps \in (0,1]}\mathbb P_{0,0}^{f_\eps}\left[\left(\sup_{0 \le s < t \le t_0}\frac{|w_t - w_s|}{|t-s|^\eta}\right)^p\right] < \infty\;,$$
	then, the invariance in law of the bridge under time-reversal and the fact $w_0=0$ under $\mathbb P_{0,0}^{f_\eps}$ yield the desired result.
	By the H\"older inequality with $q_1 \in (1,2)$ and $1/p_1 = 1 - 1/q_1$ we obtain
	\begin{align*}
		\mathbb P_{0,0}^{f_\eps}\left[\left(\sup_{0 \le s < t \le t_0}\frac{|w_t - w_s|}{|t-s|^\eta}\right)^p\right] \le A(\eps,t_0,q_1)\, \mathbb P_{0}^{f_\eps}\left[\left(\sup_{0 \le s < t \le t_0}\frac{|w_t - w_s|}{|t-s|^\eta}\right)^{pp_1}\right]^{\frac1{p_1}}\;,
	\end{align*}
	where
	\begin{align*} A(\eps,t_0,q_1) &:= \mathbb P_{0}^{f_\eps}\left[\left(\frac{\dd\mathbb P_{0,0}^{f_\eps}|_{\mathcal F_{t_0}}}{\dd\mathbb P_{0}^{f_\eps}|_{\mathcal F_{t_0}}}\right)^{q_1}\right]^{\frac1{q_1}}= \mathbb P_{0}^{f_\eps}\left[\left(\frac{p^\eps_{1-t_0}(w_{t_0},0)}{p^\eps_1(0,0)}\right)^{q_1}\right]^{\frac1{q_1}}.
	\end{align*}
	\cref{lemma:holder_tightness_unconditionned} ensures that for any $\eta\in(0,1/2)$ and any $p\ge 1$, the $(pp_1)$-th moment under $\mathbb P_{0}^{f_\eps}$ is bounded uniformly over all $\eps \in (0,1]$. To end the proof, it suffices to show that $\sup_{\eps\in(0,1]}A(\eps,t_0,q_1)<\infty$. We can write explicitly
	\begin{align*}
		A(\eps,t_0,q_1)^{q_1}&=p_1^\eps(0,0)^{-q_1}\int_{\R}p_{t_0}^\eps(0,x)\,p_{1-t_0}^\eps(x,0)^{q_1}\,\dd x\,.
	\end{align*}
	By~\eqref{eq:transition_densities} the integral appearing in the previous equation is
	\begin{equation}\label{eq:proof_tightness}
		\int_{\mathbb R} e^{(1-q_1)F_\eps(x)} g_{t_0}(x)\,\Phi_{t_0}(\eps^{-2}/2;0,x)\left(g_{1-t_0}(x)\,\Phi_{1-t_0}(\eps^{-2}/2;x,0)\right)^{q_1}\,\dd x
	\end{equation}
	and it ought to be bounded by $\eps^{2q_1}$ to cancel out the contribution of $p_1^\eps(0,0)^{-q_1}\lesssim\eps^{-2q_1}$ (see~\eqref{Eq:x=y=0}).

	We start by bounding the integral restricted to $(0,\infty)$. For $x>0$ we have
	$$ e^{(1-q_1)F_\eps(x)} = \left(\frac{\eps}{x+\eps}\right)^{q_1-1} \le\eps^{q_1-1}x^{1-q_1}\;.$$
	We recall formulae~\eqref{eq:explicit_computations_Phi} and~\eqref{Eq:R}, which imply that
	both $g_{t_0}(\cdot)\,\Phi_{t_0}(\eps^{-2}/2;0,\cdot)$ and $g_{1-t_0}(\cdot)\,\Phi_{1-t_0}(\eps^{-2}/2;\cdot,0)$ are bounded by $\eps$ times an integrable and bounded function with exponential decay at infinity.
	This means that the integral of~\eqref{eq:proof_tightness} restricted to $(0,\infty)$ is bounded by $C\eps^{q_1-1}\eps\eps^{q_1}=C\eps^{2q_1}$ as long as $x^{1-q_1}$ is integrable at $0$ i.e.~if $q_1<2$, where $C$ is a constant independent of $\eps$ which can be written as an explicit integral.
	
	We turn to bounding the integral restricted to $(-\infty,0)$. For $x\le 0$ we have
	$$ e^{(1-q_1)F_\eps(x)} = e^{(1-q_1) \frac{x}{\eps}}\;.$$
	We claim that, provided $q_1 \in (1,2)$ is chosen close enough to $1$, there exists a constant $C>0$ such that for all $x<0$ and all $\eps\in (0,1]$
	$$ e^{(1-q_1) \frac{x}{\eps}} g_{t_0}(x)\,\Phi_{t_0}(\eps^{-2}/2;0,x) \le C \epsilon^2 e^{-\frac{x^2}{8t_0}}\;,$$
	as well as
	$$ g_{1-t_0}(x)\,\Phi_{1-t_0}(\eps^{-2}/2;x,0) \le C \epsilon^2 e^{-\frac{x^2}{4(1-t_0)}}\;.$$
	Putting everything together, we deduce that the integral of~\eqref{eq:proof_tightness} on $(-\infty,0)$ is bounded by a term of order $\eps^{2(1+q_1)}$, which vanishes faster than $p_1^{\eps}(0,0)^{q_1}$ as $\eps\downarrow0$.
	
	It remains to prove the claim. The second bound is a direct consequence of~\cref{lemma:explicit_computations_Phi2}. Regarding the first bound, we use~\cref{lemma:explicit_computations_Phi2} to get
	\begin{align*}
		e^{(1-q_1) \frac{x}{\eps}} \, g_{t_0}(x)\,\Phi_{t_0}(\eps^{-2}/2;0,x)\le C \eps^2 e^{(q_1-1) \frac{|x|}{\eps}} \left( e^{-\frac{|x|}{2\sqrt{2} \eps}} + e^{- \frac{t_0}{8\eps^2}}\right) \frac{e^{-\frac{x^2}{4t_0}}}{t_0^{3/2}}\;.
	\end{align*}
	One can pick $q_1$ close enough to $1$ such that: for all $x< -1/\eps$
	$$ (q_1-1) \, \frac{|x|}{\eps} < \frac{x^2}{8t_0}\;,$$
	and for all $-1/\eps \le x <0$
	$$ (q_1-1) \, \frac{|x|}{\eps} < \min\left( \frac{|x|}{4\sqrt{2} \eps}, \frac{t_0}{16\eps^2}\right)\;.$$
	Consequently
	\begin{align*}
		e^{(1-q_1) \frac{x}{\eps}} \, g_{t_0}(x)\,\Phi_{t_0}(\eps^{-2}/2;0,x) &\le C\eps^2 \left( e^{-\frac{|x|}{4\sqrt{2} \eps}} + e^{- \frac{t_0}{16\eps^2}}\right)\frac{e^{-\frac{x^2}{8t_0}}}{t_0^{3/2}}\\
		&\le C \eps^2 e^{-\frac{x^2}{8t_0}}\;,
	\end{align*}
	thus concluding the proof of the claim. Overall, this entails that $\sup_{\eps \in (0,1]} A(\eps,t_0,q_1)<\infty$.
\end{proof}

\subsection{Estimates on the densities}\label{Subsec:EstimatesDensity}

In this subsection, we prove the estimates on the densities of the diffusion stated in~\cref{lemma:explicit_computations_Phi} and~\cref{lemma:explicit_computations_Phi2}.

\begin{proof}[Proof of~\cref{lemma:explicit_computations_Phi}]
	We rely on~\cite[Formula 1.1.5.7, pp. 167-168]{Borodin2002}\footnote{Using the notations of \cite{Borodin2002}, we rely on the formula with $r<x,\,r<z$ and the formula with $x=r,\,r<z$ which is a special instance of the first.}. In the particular case where $x=y=0$, this formula is precisely~\eqref{Eq:x=y=0}. In the case where $x\ge 0$, $y\ge 0$ and $x+y>0$, after a disintegration along the marginal at time $t$ of the Brownian motion, this formula writes
	\begin{align}\label{eq:borodin}
		g_t(y-x)\,\Phi_t(\lambda;x,y)=&\sqrt{\frac{2}{\pi t}}\,e^{-\frac{x^2+y^2}{2t}}\sinh\left(\frac{xy}{t}\right)+\frac{x+y}{2\pi} I_t\,,
	\end{align}
	where
	$$ I_t := \int_0^t\frac{1-e^{-\lambda(t-s)}}{\lambda (t-s)^{3/2}}\frac{1}{s^{3/2}}\,e^{-\frac{(x+y)^2}{2s}}\, \dd s\,.$$
	We write $I_t$ as the sum of $I_t^{(1)}$, $I_t^{(2)}$ and $I_t^{(3)}$ where
	\begin{equation}\label{eq:It12}
		\begin{split}
			I_t^{(1)} &:= \frac{1}{t^{3/2}}\,e^{-\frac{(x+y)^2}{2t}}\int_0^t\frac{1-e^{-\lambda(t-s)}}{\lambda (t-s)^{3/2}}\, \dd s\,,\\
			I_t^{(2)} &:= \int_0^{t/2}\frac{1-e^{-\lambda(t-s)}}{\lambda (t-s)^{3/2}}\left(\frac{1}{s^{3/2}}\,e^{-\frac{(x+y)^2}{2s}}-\frac{1}{t^{3/2}}\,e^{-\frac{(x+y)^2}{2t}}\right)\dd s\,,\\
			I_t^{(3)} &:= \int_{t/2}^t\frac{1-e^{-\lambda(t-s)}}{\lambda (t-s)^{3/2}}\left(\frac{1}{s^{3/2}}\,e^{-\frac{(x+y)^2}{2s}}-\frac{1}{t^{3/2}}\,e^{-\frac{(x+y)^2}{2t}}\right)\dd s\,.
		\end{split}
	\end{equation}
	We first concentrate on $I_t^{(1)}$, and write
	\begin{equation}\label{eq:transition_densities_bound1}
		\int_0^t \frac{1-e^{-\lambda(t-s)}}{\lambda (t-s)^{3/2}}\, \dd s=\int_0^t\frac{1-e^{-\lambda s}}{\lambda s^{3/2}}\, \dd s=\int_0^\infty\frac{1-e^{-\lambda s}}{\lambda s^{3/2}}\, \dd s-\int_t^{\infty}\frac{1-e^{-\lambda s}}{\lambda s^{3/2}}\, \dd s\,.
	\end{equation}
	This is the sum of two terms, the first term satisfies
	\begin{align*}
		\int_0^\infty\frac{1-e^{-\lambda s}}{\lambda s^{3/2}}\, \dd s &= \frac{1}{\lambda^{1/2}}\int_0^\infty\frac{1-e^{-u}}{u^{3/2}}\, \dd u= \frac{2}{\lambda^{1/2}}\int_0^{\infty}\frac{e^{-u}}{u^{1/2}}\, \dd u=\frac{2\sqrt{\pi}}{\lambda^{1/2}}\,.
	\end{align*}
	The second term is given by
	\begin{align*}
		\int_t^{\infty}\frac{1-e^{-\lambda s}}{\lambda s^{3/2}}\, \dd s &\le \frac{1}{\lambda}\int_{t}^\infty\frac{1}{s^{3/2}}\, \dd s = \frac{2}{\lambda \sqrt{t}}\,.
	\end{align*}
	Consequently
	\begin{equation}
		\frac{x+y}{2\pi} I_t^{(1)} = \frac1{\sqrt{\lambda}} \frac{x+y}{t^{3/2}\sqrt{\pi}}\,e^{-\frac{(x+y)^2}{2t}} + \frac1{\lambda} R^{(1)}(\lambda,t,x+y)\;,
	\end{equation}
	with
	$$ R^{(1)}(\lambda,t,x+y) := \frac{x+y}{2 \pi t^{3/2}}\,e^{-\frac{(x+y)^2}{2t}} \int_t^{\infty}\frac{1-e^{-\lambda s}}{s^{3/2}}\;,$$
	and
	$$ |R^{(1)}(\lambda,t,x+y)| \le \frac{x+y}{\pi t^2}\,e^{-\frac{(x+y)^2}{2t}} \le C \frac1{t^{3/2}} e^{-\frac{(x+y)^2}{4t}}\,.$$
	We turn to bounding $I_t^{(3)}$. To that end, we set $\phi(u):=u^{-3/2}\,e^{-\frac{1}{2u}}$ and observe that
	\begin{align*}
		I_t^{(3)} = \frac1{(x+y)^3} \int_{t/2}^t\frac{1-e^{-\lambda(t-s)}}{\lambda (t-s)^{3/2}}\int_{\frac{s}{(x+y)^2}}^{\frac{t}{(x+y)^2}}-\phi'(u)\,\dd u\,\dd s\,.
	\end{align*}
	The derivative is explicit and satisfies for some constant $C>0$
	$$ |-\phi'(u)| \le C e^{-\frac{1}{2u}}\left(u^{-7/2} \vee u^{-5/2}\right)\,,$$
	uniformly over all $u>0$, whence, because $s\ge t/2$, bounding the derivative,
	\[
	\left|\int_{\frac{s}{(x+y)^2}}^{\frac{t}{(x+y)^2}}-\phi'(u)\,\dd u\right| \le C \frac{(t-s)}{(x+y)^2} e^{-\frac{(x+y)^2}{2t}} \left( \left(\frac{(x+y)^2}{t}\right)^{7/2} \vee \left(\frac{(x+y)^2}{t}\right)^{5/2}\right)\,.
	\]
	Thus
	\begin{align*}
		\frac{x+y}{2\pi}|I_t^{(3)}| &\le \frac{C}{\lambda} \frac1{(x+y)^4} \left( \left(\frac{(x+y)^2}{t}\right)^{7/2} \vee \left(\frac{(x+y)^2}{t}\right)^{5/2}\right) e^{-\frac{(x+y)^2}{2t}} \int_{t/2}^t\frac{\dd s}{(t-s)^{1/2}}\\
		&\le \frac{C}{\lambda} \frac{\sqrt{t}}{(x+y)^4} \left( \left(\frac{(x+y)^2}{t}\right)^{7/2} \vee \left(\frac{(x+y)^2}{t}\right)^{5/2}\right) e^{-\frac{(x+y)^2}{2t}}\\
		&\le \frac{C}{\lambda t^{3/2}} \left( \left(\frac{(x+y)^2}{t}\right)^{3/2} \vee \left(\frac{(x+y)^2}{t}\right)^{1/2}\right) e^{-\frac{(x+y)^2}{2t}}\;.
	\end{align*}
	Playing around with the constant in the exponential decay rate in order to get rid of the polynomial prefactors, we get the further bound
	$$ \frac{x+y}{2\pi}|I_t^{(3)}| \le  \frac{C}{\lambda t^{3/2}} e^{-\frac{(x+y)^2}{4t}}\;.$$
	We turn to $I_t^{(2)}$ and write
	\begin{equation}\label{Eq:It2}\begin{split}
			|I_t^{(2)}| &\le \int_0^{t/2}\frac{1-e^{-\lambda(t-s)}}{\lambda (t-s)^{3/2}} \frac{1}{s^{3/2}}\,e^{-\frac{(x+y)^2}{2s}}\, \dd s\\
			&+ \int_0^{t/2}\frac{1-e^{-\lambda(t-s)}}{\lambda (t-s)^{3/2}} \frac{1}{t^{3/2}}\,e^{-\frac{(x+y)^2}{2t}}\, \dd s\;.
		\end{split}	
	\end{equation}
	First observe that there exists $C>0$ such that for all $z>0$
	\begin{equation}\label{Eq:tricky}
		\int_{z}^{\infty}u^{-1/2}\,e^{-u}\,\dd u \le \1_{(z\le 1)} \int_{0}^{\infty}u^{-1/2}\,e^{-u}\,\dd u + \1_{(z>1)} e^{-z} \le C e^{-z}\,.
	\end{equation}
	The first term of~\eqref{Eq:It2} is 
	\begin{align*}
		\int_0^{t/2}\frac{1-e^{-\lambda(t-s)}}{\lambda(t-s)^{3/2}}\frac{1}{s^{3/2}}\,e^{-\frac{(x+y)^2}{2s}}\,\dd s&\le C\lambda^{-1}t^{-3/2}\int_0^{t/2}\frac{1}{s^{3/2}}\,e^{-\frac{(x+y)^2}{2s}}\,\dd s\\
		&\le C\lambda^{-1}t^{-3/2}(x+y)^{-1}\int_{e^{-\frac{(x+y)^2}{t}}}^{\infty}u^{-1/2}\,e^{-u}\,\dd u\\
		&\le C\lambda^{-1}t^{-3/2}(x+y)^{-1}\,e^{-\frac{(x+y)^2}{t}}
	\end{align*}
	where we used the change of variables $u=(x+y)^2/(2s)$ at the second line. The second term of~\eqref{Eq:It2} is
	\begin{align*}
		&\int_0^{t/2}\frac{1-e^{-\lambda(t-s)}}{\lambda(t-s)^{3/2}}\frac{1}{t^{3/2}}\,e^{-\frac{(x+y)^2}{2t}}\,\dd s\\
		&\le \lambda^{-1}t^{-3/2}\,e^{-\frac{(x+y)^2}{2t}}\int_0^{t/2}\frac{1}{(t-s)^{3/2}}\, \dd s \le C \lambda^{-1}t^{-2}\,e^{-\frac{(x+y)^2}{2t}}\,.
	\end{align*}
	Thus we have shown that
	\begin{align*}
		\frac{x+y}{2\pi} \left| I_t^{(2)} \right| &\le \frac{C}{\lambda} \left(\frac1{t^{3/2}} e^{-\frac{(x+y)^2}{t}} + \frac{x+y}{t^2}\,e^{-\frac{(x+y)^2}{2t}}\right)\\
		&\le  \frac{C}{\lambda t^{3/2}} e^{-\frac{(x+y)^2}{4t}}\;.
	\end{align*}
	Consequently
	$$\frac{x+y}{2\pi} \left(I_t^{(2)} + I_t^{(3)}\right) =: \lambda^{-1} R^{(2-3)}(\lambda,t,x+y)\,,$$
	and
	\begin{align*}
		|R^{(2-3)}(\lambda,t,x+y)| &\le \frac{C}{t^{3/2}} e^{-\frac{(x+y)^2}{4t}}\;.
	\end{align*}
	Putting everything together, we obtain
	\begin{align*}
		g_t(y-x)\,\Phi_t(\lambda;x,y)=& \frac{\sqrt{2}}{\sqrt{\pi t}}\,e^{-\frac{x^2+y^2}{2t}}\sinh\left(\frac{xy}{t}\right)+\frac1{\sqrt{\lambda}} \frac{x+y}{t^{3/2}\sqrt{\pi}}\,e^{-\frac{(x+y)^2}{2t}}+ \frac1{\lambda}R(\lambda,t,x+y)\,
	\end{align*}
	where $R(\lambda,t,x+y) = R^{(1)}(\lambda,t,x+y) + R^{(2-3)}(\lambda,t,x+y)$.
\end{proof}

\begin{proof}[Proof of~\cref{lemma:explicit_computations_Phi2}]
	Since $\Phi_t(\lambda; x,y) = \Phi_t(\lambda;y,x)$, we have
	$$ g_t(x)\,\Phi_t(\lambda;x,0) = g_t(x)\,\Phi_t(\lambda;0,x)\;,$$
	and it suffices to deal with, say, the second expression. We rely on~\cite[Formula 1.1.5.7, pp. 161-162]{Borodin2002}\footnote{Using the notations of \cite{Borodin2002}, we rely on the formula with $x=r,\,z<r$.}. It implies that for $x<0$, for $t>0$, for $\lambda>0$,
	\[
	g_t(x)\,\Phi_t(\lambda;0,x)=\frac{-x}{2\pi}\int_0^t\frac{1-e^{-\lambda(t-s)}}{\lambda(t-s)^{3/2}}\frac{e^{-\lambda s}}{s^{3/2}}\,e^{-\frac{x^2}{2s}}\, \dd s\,.
	\]
	This formula differs from the case $x>0$ by an exponential factor. The latter allows for easy bounds. 
	
	We split the integral into the sum of the integrals over $[0,t/2]$ and $[t/2,t]$. On the first interval, we bound it as follows
	\begin{align*}
		\int_{0}^{t/2}\frac{1-e^{-\lambda(t-s)}}{\lambda(t-s)^{3/2}}\frac{e^{-\lambda s}}{s^{3/2}}\,e^{-\frac{x^2}{2s}}\, \dd s &\le C \frac1{\lambda t^{3/2}} \int_0^{t/2}s^{-3/2}\,e^{-\frac{x^2}{4s}} e^{-\lambda s - \frac{x^2}{4s}}\,\dd s\;.
	\end{align*}
	At this point, we let $s_0>0$ be such that $\lambda s_0 =  \frac{x^2}{4s_0}$ and we note that for all $s>0$
	$$ \lambda s + \frac{x^2}{4s} \ge \lambda s_0 = \sqrt{\lambda} \frac{|x|}{2}\;.$$
	Consequently we get
	\begin{align*}
		\int_{0}^{t/2}\frac{1-e^{-\lambda(t-s)}}{\lambda(t-s)^{3/2}}\frac{e^{-\lambda s}}{s^{3/2}}\,e^{-\frac{x^2}{2s}}\, \dd s &\le C \frac1{\lambda t^{3/2}} e^{-\sqrt{\lambda} \frac{|x|}{2}} \int_0^{t/2}s^{-3/2}\,e^{-\frac{x^2}{4s}}\,\dd s\\
		&\le C \frac1{\lambda t^{3/2}} e^{-\sqrt{\lambda} \frac{|x|}{2}} |x|^{-1}\int_{\frac{x^2}{2t}}^{\infty}u^{-1/2}\,e^{-u}\,\dd u\\
		&\le C \frac1{\lambda t^{3/2}} e^{-\sqrt{\lambda} \frac{|x|}{2}} |x|^{-1}\,e^{-\frac{x^2}{2t}}\,,
	\end{align*}
	where we used~\eqref{Eq:tricky} at the last line. 
	On the second interval, we bound the integral as follows
	\[
	\int_{t/2}^{t}\frac{1-e^{-\lambda(t-s)}}{\lambda(t-s)^{3/2}}\frac{e^{-\lambda s}}{s^{3/2}}\,e^{-\frac{x^2}{2s}}\, \dd s \le C \frac{e^{-\lambda \frac{t}{2}}}{t^{3/2}}  e^{-\frac{x^2}{2t}} \int_{t/2}^t\frac{1}{(t-s)^{1/2}}\, \dd s \le C \frac{e^{-\lambda \frac{t}{2}}}{t} e^{-\frac{x^2}{2t}}\;.
	\]
	Tuning the exponential factor in $\lambda t$, we further get
	\[
	\int_{t/2}^{t}\frac{1-e^{-\lambda(t-s)}}{\lambda(t-s)^{3/2}}\frac{e^{-\lambda s}}{s^{3/2}}\,e^{-\frac{x^2}{2s}}\, \dd s \le C \frac{e^{-\lambda \frac{t}{4}}}{\lambda t^2} e^{-\frac{x^2}{2t}}\;.
	\]
	Putting these two bounds together, and tuning the constant in the exponential factor in order to get rid of a factor $|x| / t^{1/2}$, we obtain
	\begin{align*}
		g_t(x)\,\Phi_t(\lambda;0,x) &\le C\left( \frac1{\lambda t}\,e^{-\sqrt{\lambda} \frac{|x|}{2}} + \frac{|x|}{\lambda t^{3/2}}\,e^{-\lambda \frac{t}{4}}\right)\frac{e^{-\frac{x^2}{2t}}}{t^{1/2}} \\
		&\le C \left(e^{-\sqrt{\lambda} \frac{|x|}{2}} + e^{-\lambda \frac{t}{4}}\right)\frac{e^{-\frac{x^2}{4t}}}{\lambda t^{3/2}}\;.
	\end{align*}
	The proof is complete.
\end{proof}

\section{Formulations of the SPDEs and an approximation result}\label{section:SPDEs}

We recall that a random process $W:\,\Omega\times L^2([0,1])\times[0,\infty)\to \R$ is called white noise if for every $\phi\in L^2([0,1])$, $(W_t(\phi))_{t\ge0}$ has continuous sample paths and $W$ is a centred Gaussian field with covariance 
\begin{equation*}
	\mathbb E[W_s(\phi)W_t(\phi')]=s\wedge t\,(\phi,\phi')_{L^2}\,.
\end{equation*}

We denote by $(p_t(x,y))_{t\ge0,x,y\in[0,1]}$ the heat kernel with Dirichlet boundary conditions on $[0,1]$ and by $(P_t)_{t\ge0}$ the corresponding semigroup. Similarly, we denote by $(g_t(x,y))_{t\ge0,x,y\in\R}$ the heat kernel on $\R$ and by $(G_t)_{t\ge0}$ the corresponding semigroup.

We set
\begin{equation}\label{eq:stochastic_convolution}
	V_t(x):=\int_0^t\int_0^1p_{t-r}(x,y)\,W(\dd r,\dd y),\quad t\ge0,\,x\in[0,1]\,.
\end{equation}
The integral is in the sense of Wiener ($\int_0^t\int_0^1p_{t-r}(x,y)^2\,\dd y\,\dd r=\int_0^tp_{2r}(x,x)\,\dd r<\infty$). The process $V$ is the solution of the stochastic heat equation \eqref{eq:SPDE_skew} with $\kappa=0$.

We are interested in the skew stochastic heat equation
\begin{equation}\label{eq:SHE}
	\left\{
	\begin{aligned}
		&\partial_t u_t(x)=\frac12 \partial^2_{xx}u_t(x) +b(u_t(x))+\dot{W}_{t}(x),\quad x\in[0,1],\\
		&u_t(0)=u_t(1)=0\,,\quad t\ge0\,,
	\end{aligned}
	\right.
\end{equation}
i.e. in the case where $b=\kappa\delta_0$ with $\kappa>0$.

\subsection{Formulation of the skew equation}

The following is a classical existence and uniqueness result for~\eqref{eq:SHE} in the case where $b$ is a Lipschitz continuous function:
\begin{theorem}[{\cite[Prop. C.2.1]{Dalang2026}}]\label{theorem:existence_uniqueness_smooth}
	Let $u_0\in L^2(0,1)$ and $T>0$. If $b:\R\to\R$ is Lipschitz continuous, there exists a unique $u\in\mathcal C([0,T],L^2(0,1))$ such that
	\begin{enumerate}
		\item For all $\phi \in \mathcal C^{\infty}([0,1])$ with $\phi(0)=\phi(1)=0$, for all $t\in[0,T]$,
			\begin{equation}\label{eq:she_weak}
				(u_t,\phi)_{L^2}=(u_0,\phi)_{L^2}+\int_0^t\left[(u_s,\frac12\phi'')_{L^2}+(b(u_s),\phi)_{L^2}\right]\dd s+W_t(\phi)\,.
			\end{equation}
		\item For all $t\in[0,T]$, for all $x\in[0,1]$,
			\begin{equation}\label{eq:she_mild}
				u_t(x)=P_tu_0(x)+\int_0^t\int_0^1p_{t-s}(x,y)b(u_s)(y)\,\dd y\,\dd s+V_t(x)\,.
			\end{equation}
	\end{enumerate}
	In particular,~\eqref{eq:SHE} admits both a unique weak solution and a unique mild solution, and they coincide.
\end{theorem}

The following two definitions and theorem are borrowed from the forthcoming work \cite{Butkovsky2026}. They adapt some of the results of \cite{Athreya2023,Athreya2025} to the case of Dirichlet boundary conditions.

\begin{definition}
	Let $b\in\mathcal C^{-1}(\R)$. We say that a sequence of functions $(b^n)_{n\ge0}$ converges to $b$ in $\mathcal C^{-1-}(\R)$ as $n\rightarrow\infty$ if $\sup_{n\ge0}\|b^n\|_{\mathcal C^{-1}}<\infty$ and
	\begin{equation*}
		\underset{n\rightarrow\infty}{\lim}\|b^n-b\|_{\mathcal C^{-1-\eps}}=0\,,\quad\textrm{for any}\,\eps>0\,.
	\end{equation*}
	Notice that the sequence $b^n:=\kappa g_{1/n}$ converges to $\kappa \delta_0$ in $\mathcal C^{-1-}$.
\end{definition}

\begin{definition}[Solution of the skew stochastic heat equation]\label{Def:sshe}

Let $u_0\in\mathcal C([0,1])$. A measurable adapted process $u:\,\Omega\times[0,\infty)\times[0,1]\rightarrow\R$ is called a \emph{mild solution} of \eqref{eq:SPDE_skew}
with initial condition $u_0$ if there exists a process $K:\,\Omega\times[0,\infty)\times[0,1]\rightarrow\R$ such that
\begin{enumerate}[(1)]
	\item For every $x\in[0,1]$ and every $t\ge0$,
	\begin{equation}\label{eq:sshe}
		u_t(x)=P_tu_0(x)+K_t(x)+\int_0^t\int_0^1p_{t-r}(x,y)W(\dd r,\dd y)\,,
	\end{equation}
   \item for any sequence of functions $(b^n)_{n\ge 0}$ in $\mathcal C_b^\infty$ such that $b^n\to b$ in $\mathcal C^{-1-}$, we have for any $T>0$,
		\begin{equation}\label{Kcurlydef}
		\sup_{t\in[0,T]}\sup_{x\in[0,1]}\left|\int_0^t\int_0^1p_{t-r}(x,y) b^n(u_r(y))\,\dd y\,\dd r- K_t(x)\right|\to0\quad  \text{in probability as $n\to\infty$.}
		\end{equation}
	\item a.s. the function $u$ is continuous on $[0,\infty)\times[0,1]$.
  \end{enumerate}
\end{definition}

\begin{theorem}[\cite{Butkovsky2026}]\label{theorem:properties_spde_dirac}
	Let $b\in\mathcal C^{-1}$ and $u_0\in\mathcal C([0,1])$. Then:
	\begin{enumerate}
		\item\label{item:properties_spde_dirac1} The SPDE~\eqref{eq:SHE} has a unique solution in law (in the sense of~\cref{Def:sshe}).
		\item\label{item:properties_spde_dirac2} The unique solution is a Markov process with state space $\mathcal C([0,1])$.
		\item\label{item:properties_spde_dirac3} (stability) If $u^n$ is a solution to~\eqref{eq:SHE} (in the sense of~\cref{theorem:existence_uniqueness_smooth}) with drift $b^n\in\mathcal C_b^{\infty}$ and starting from $u_0$, and if $b^n$ converges to $b$ in $\mathcal C^{-1-}$,
	then we have the convergence in distribution
	\begin{equation*}
		u^n\rightarrow u\,,\quad\textrm{as }\,n\rightarrow\infty\,,
	\end{equation*}
	in $\mathcal C([0,\infty)\times[0,1])$ endowed with the topology of uniform convergence on compact sets.
	\end{enumerate}
	
\end{theorem}
\subsection{Approximation result and proof of~\cref{theorem:invariant_measure}}

Our proof of~\cref{theorem:invariant_measure} relies on an approximation result under the stationary measures. This approximation result is an extension of the stability result (\cref{theorem:properties_spde_dirac},~\cref{item:properties_spde_dirac3}): instead of starting $u^{n}$ and $u$ from the \emph{same, deterministic} initial condition $u_0$, we start these two processes from their respective invariant measures, provided the sequence $(b^n)_n$ is conveniently chosen.

Let $\kappa>0$, and let
\begin{equation}\label{eq:approximate_b}
	b^n=\kappa g_{1/n}\,,\quad n\ge0\,,
\end{equation}
where $g_{t}$ is the density of a random variable with law $\mathcal N(0,t)$. We already noticed that $b^n$ converges to $\kappa\delta_0$ in $\mathcal C^{-1-}$.
For every $n\ge0$, let $B^n$ be the antiderivative of $b^n$ defined by 
\begin{equation}\label{eq:anditerivatives}
	B^n(x)=-\int_x^{\infty}b^n(y)\,\dd y\,,\quad x\in\R\,.
\end{equation}
As for every $n\ge0$, $b^n$ is ($\kappa$ times) a smooth probability density, $B^n$ is a nonpositive, nondecreasing function, bounded below by $-\kappa$. We define the probability measure on $\mathcal C([0,1])$
\begin{equation}\label{eq:approx_reversible_measure}
	\mathbb Q^{\kappa,n}_{0,0}(\dd w):=Z_{\kappa,n}^{-1}\exp\left(2\int_0^1B^n(w_s)\,\dd s\right)\mathbb W_{0,0}(\dd w)\,.
\end{equation}

\begin{remark}\label{remark:approximation}
	The sequence $b^n$ is not explicit in the notation $\mathbb Q^{\kappa,n}_{0,0}$, but will always be given by~\eqref{eq:approximate_b}.
	We could have taken any sequence $b^n$ which converges to $\kappa\delta_0$ in $\mathcal C^{-1-}$ as $n\rightarrow\infty$, such that, uniformly in $n\ge0$, the antiderivatives given by~\eqref{eq:anditerivatives} are uniformly bounded above and below.
	Indeed, this implies that the Radon-Nikodym derivatives
	\begin{equation*}
		\frac{\dd \mathbb Q_{0,0}^{\kappa,n}}{\dd \mathbb W_{0,0}}
	\end{equation*}
	are uniformly bounded in $n$. In our case, they are uniformly bounded by $\exp(2\kappa)$.
\end{remark}

\begin{remark}\label{remark:reversible_measure_spde}
	It is a classical result that $\mathbb Q^{\kappa,n}_{0,0}$ is the only reversible probability measure of the solution of~\eqref{eq:SHE} with $b=b^n$ (see for instance~\cite[Th. 5.3.11, p. 279]{Dalang2026}).
\end{remark}

\begin{proposition}\label{prop:stability}
	Let $(b^n)$ be the sequence of functions given by~\eqref{eq:approximate_b}. Let $u^{\kappa,n}$ be the solution to~\eqref{eq:SHE} (in the sense of~\cref{theorem:existence_uniqueness_smooth}) with drift $b^n$ and started from $u^{\kappa,n}_0\eqlaw\mathbb Q_{0,0}^{\kappa,n}$ (given by~\eqref{eq:approx_reversible_measure}) taken independent from $W$. Then one has the convergence in distribution in the space $\mathcal C([0,\infty)\times[0,1])$ endowed with the topology of uniform convergence on compact sets
	\begin{equation*}
		u^{\kappa,n}\rightarrow u^\kappa\,,\quad\textrm{as }\,n\rightarrow\infty\,,
	\end{equation*}
	where $u^\kappa$ is the unique solution to~\eqref{eq:SHE} (in the sense of~\cref{Def:sshe}) with drift $b=\kappa\delta_0$ started from $u_0^\kappa\eqlaw\mathbb Q_{0,0}^\kappa$.
\end{proposition}

\noindent As a direct corollary, we obtain the proof of~\cref{theorem:invariant_measure}.
\begin{proof}[Proof of~\cref{theorem:invariant_measure}]
	In this proof, we take $(b^n)$ the sequence given by~\eqref{eq:approximate_b}. We need to prove that for every bounded measurable $f_1, f_2:\mathcal C([0,\infty)\times[0,1])\to\R$ and for every $t>0$,
	\begin{equation}\label{eq:reversible_measure}
		\int\mathbb Q^{\kappa}_{0,0}(\dd u_0)\,f_1(u_0) \, \mathbb E[f_2(u_t^{\kappa}(u_0))]=\int\mathbb Q^{\kappa}_{0,0}(\dd u_0)\,\mathbb E[f_1(u_t^{\kappa}(u_0))]
		 \, f_2(u_0)\,.
	\end{equation}
	It follows by~\cref{remark:reversible_measure_spde} that for all $t>0$, for all $n\ge0$
	\begin{equation}\label{eq:QKN_reversible}
		\int\mathbb Q^{\kappa,n}_{0,0}(\dd u_0)\,f_1(u_0) \, \mathbb E[f_2(u_t^{\kappa,n}(u_0))]=\int\mathbb Q^{\kappa,n}_{0,0}(\dd u_0)\,\mathbb E[f_1(u_t^{\kappa,n}(u_0))] \, f_2(u_0)\,.
	\end{equation}
	By~\cref{prop:stability}, see in particular~\eqref{eq:stability1} with $F(u)=f_1(u_0)f_2(u_t)$, we can pass to the limit in~\eqref{eq:QKN_reversible} and obtain~\eqref{eq:reversible_measure}.
\end{proof}

\begin{proof}[Proof of~\cref{prop:stability}]
	The result is a consequence of~\cref{theorem:properties_spde_dirac},~\cref{item:properties_spde_dirac3} and the Dominated Convergence Theorem.
	Let $F$ be any bounded continuous function on $\mathcal C([0,\infty)\times[0,1])$. For any $u_0\in \mathcal C([0,\infty)\times[0,1])$, as $n\rightarrow\infty$, $\mathbb E[F(u^{\kappa,n}(u_0))]$ goes to $\mathbb E[F(u^{\kappa}(u_0))]$ by~\cref{theorem:properties_spde_dirac}, and the sequence is uniformly bounded in $n$ ($F$ is bounded). We claim that for $\mathbb W_{0,0}$-almost every $u_0$, as $n\rightarrow\infty$,
	\begin{equation*}
		\frac{\dd \mathbb Q_{0,0}^{\kappa,n}}{\dd \mathbb W_{0,0}}(u_0)\longrightarrow\frac{\dd \mathbb Q_{0,0}^{\kappa}}{\dd \mathbb W_{0,0}}(u_0)
	\end{equation*}
	and that the L.H.S.~is bounded uniformly over all $u_0\in \mathcal C([0,\infty)\times[0,1])$ and all $n\ge 1$. Then, we deduce from the Dominated Convergence Theorem for $\mathbb W_{0,0}$ that
	\begin{equation}\label{eq:stability1}
		\int\mathbb Q_{0,0}^{\kappa,n}(\dd u_0)\,\mathbb E[F(u^{\kappa,n}(u_0))]=\int\mathbb W_{0,0}(\dd u_0)\,\frac{\dd \mathbb Q_{0,0}^{\kappa,n}}{\dd \mathbb W_{0,0}}(u_0)\,\mathbb E[F(u^{\kappa,n}(u_0))]\,,
	\end{equation}
	converges as $n\to\infty$ to
	\begin{equation*}
		\int\mathbb W_{0,0}(\dd u_0)\,\frac{\dd \mathbb Q_{0,0}^{\kappa}}{\dd \mathbb W_{0,0}}(u_0)\,\mathbb E[F(u^{\kappa}(u_0))]=\int\mathbb Q_{0,0}^{\kappa,n}(\dd u_0)\,\mathbb E[F(u^{\kappa}(u_0))]\,.
	\end{equation*}
	This establishes the asserted result.
	
	It remains to prove the claim. Recall that
	\begin{equation*}
		\frac{\dd \mathbb Q_{0,0}^{\kappa,n}}{\dd \mathbb W_{0,0}}(w) =\frac{\exp\left(2\int_0^1B^n(w_s)\,\dd s\right)}{Z_{\kappa,n}}, \ \mbox{ with }\quad Z_{\kappa,n} = \mathbb W_{0,0}\left[\exp\left(2\int_0^1B^n(w_s)\,\dd s\right)\right],
	\end{equation*}
	and similarly,
	\begin{equation*}
		\frac{\dd \mathbb Q_{0,0}^{\kappa}}{\dd \mathbb W_{0,0}}(w) =\frac{\exp\left(-2\kappa \int_0^1 \1_{\R_-}(w_s)\dd s\right)}{Z_{\kappa}}, \ \mbox{ with }\ Z_{\kappa} = \mathbb W_{0,0}\left[\exp\left(-2\kappa \int_0^1 \1_{\R_-}(w_s)\dd s\right)\right]\!.
	\end{equation*}	
	As $n\rightarrow\infty$, $B^n$ goes to $-\kappa\1_{\R_-}$ except at $0$. Indeed, for $x\neq0$, by a change of variables,
	\begin{equation*}
		\int_x^\infty b^n(y)\,\dd y=\kappa\int_x^\infty g_{1/n}(y)\,\dd y=\kappa\int_{\sqrt{n} x}^{\infty}g_1(y)\,\dd y\,.
	\end{equation*}
	By the Dominated Convergence theorem, if $x<0$, the previous integral goes to $\kappa$, and if $x>0$, it goes to $0$.
	Consequently, as $B^n$ is uniformly bounded in $n$, and because for $\mathbb W_{0,0}$-almost every $w$, $\textrm{Leb}(\{t\in[0,1]:\,w_t=0\})=0$, by the Dominated Convergence Theorem, for $\mathbb W_{0,0}$-almost every $w$, as $n$ goes to $\infty$,
	\begin{equation*}
		\int_0^1B^n(w_s)\,\dd s\longrightarrow-2\kappa\int_0^1\1_{\R_-}(w_s)\,\dd s\,.
	\end{equation*}
	As, for every $n\ge0$, $B^n\le0$, then for every $w\in\mathcal C([0,1])$ and $n\ge0$,
	\begin{equation*}
		\exp\left(2\int_0^1B^n(w_s)\,\dd s\right)\le 1\,.
	\end{equation*}
	As a consequence, by the Dominated Convergence Theorem for $\mathbb W_{0,0}$,
	\begin{equation*}
		Z_{\kappa,n} = \mathbb W_{0,0}\left[\exp\left(2\int_0^1B^n(w_s)\,\dd s\right)\right]\longrightarrow\mathbb W_{0,0}\left[\exp\left(-2\kappa\int_0^1\1_{\R_-}(w_s)\,\dd s\right)\right]=Z_\kappa\,.
	\end{equation*}
	Finally, as for every $n\ge0$, $-\kappa\le B^n$, then
	\begin{equation*}
		\mathbb W_{0,0}\left[\exp\left(2\int_0^1B^n(w_s)\,\dd s\right)\right]\ge\exp(-2\kappa)\,.
	\end{equation*}
	This implies that for $\mathbb W_{0,0}$-almost every $u_0$, as $n\rightarrow\infty$,
	\begin{equation*}
		\frac{\dd \mathbb Q_{0,0}^{\kappa,n}}{\dd \mathbb W_{0,0}}(u_0)\longrightarrow\frac{\dd \mathbb Q_{0,0}^{\kappa}}{\dd \mathbb W_{0,0}}(u_0)
	\end{equation*}
	and uniformly in $n\ge0$ and $u_0$,
	\begin{equation*}
		\frac{\dd \mathbb Q_{0,0}^{\kappa,n}}{\dd \mathbb W_{0,0}}(u_0)\le\exp(2\kappa)\,,
	\end{equation*}
	and the claim is proved.
\end{proof}

\subsection{Formulation of the reflected equation}

\begin{definition}[Solution of the reflected stochastic heat equation]\label{Def:NP}
	A pair $(u,\eta)$ is said to be a solution of \eqref{eq:SPDE_NP} if:
	\begin{enumerate}[(i)]
		\item \label{item:NP_1} $u$ is a nonnegative, continuous process with $u(t,0)=u(t,1)=0$, for all $t\ge0$ a.s.
		\item \label{item:NP_2} $\eta(\dd t,\dd x)$ is a random measure on $[0,\infty)\times(0,1)$ such that $\eta([0,T]\times(\eps,1-\eps))<\infty$ for all $\eps>0$ and for all $T>0$ a.s.
		\item \label{item:NP_3} For all $t\ge0$ and all $\phi\in\mathcal C^\infty([0,1])$ such that $\phi(0)=\phi(1)=0$,
			\begin{align*}
				\mathrm{a.s.},\quad(u_t,\phi)_{L^2}=(u_0,\phi)_{L^2}+\int_0^t(u_s,\frac12\phi'')_{L^2}\,\dd s+\int_0^t\int_0^1\phi(x)\,\eta(\dd s,\dd x)+W_t(\phi)\,.
			\end{align*}
		\item \label{item:NP_4} $\int u\,\dd\eta=0$ a.s.
	\end{enumerate}
\end{definition}

\begin{remark}
	The mild formulation for~\eqref{eq:SPDE_NP} is mentioned in~\cite[Remark, p.19]{DonatiMartin1993}, however both the original papers~\cite{Nualart1992} and~\cite{DonatiMartin1993} only rely on the weak formulation.
\end{remark}

\section{Proof of~\cref{theorem:dynamical_convergence}}\label{Sec:DynCV}

This section is devoted to the proof of~\cref{theorem:dynamical_convergence}, namely the convergence in law of $u^\kappa$ to $u^+$ under their respective stationary measures. In the first subsection, we establish tightness and in the second, we identify the limit.

\subsection{Tightness estimates}\label{section:tightness}

We recall that $u^{\kappa}$ denotes the solution of~\eqref{eq:SPDE_skew} in the sense of~\cref{Def:sshe} started from $\mathbb Q_{0,0}^{\kappa}$.
We herein prove the
\begin{proposition}\label{prop:tightness}
	The family $(u^\kappa)_{\kappa>1}$ is tight in $\mathcal C([0,\infty)\times[0,1])$.
\end{proposition}
In order to prove the result, we make use of the stationarity.
Specifically, we apply the Lyons-Zheng decomposition to some approximation sequence given by~\cref{prop:stability}, and we conclude by an interpolation argument.

For any $\eta>0$ and $r\ge1$, we introduce the Sobolev-Slobodeckij space:

	\begin{equation*}
		W^{\eta,r}(0,1):=\left\{f:\,\int_{[0,1]}|f(x)|^r\,\dd x+\int_{[0,1]^2}\frac{|f(x)-f(y)|^r}{|x-y|^{\eta r+1}}\,\dd x\dd y<\infty\right\}\,.
	\end{equation*}

	\noindent Let $(e_k)_{k\ge1}$ be a Hilbert basis of $H=L^2(0,1)$ made up of eigenfunctions of the Dirichlet Laplacian on $[0,1]$. For any $\beta>0$, we introduce the Sobolev space of distributions
	\begin{equation*}
		H^{-\beta}(0,1):=\left\{f\in\mathcal S'(0,1):\,\|f\|_{-\beta}^2:=\sum_{k\ge1}k^{-2\beta}|\langle f,e_k\rangle|^2<\infty\right\}\,,
	\end{equation*}
	where $\mathcal S'(0,1)$ is the set of distributions on $(0,1)$ and $\langle \cdot , \cdot \rangle$ is the pairing of distributions. The following result is the content of~\cite[Lemma 16]{Labb2018}, the proof presented therein is based on interpolation inequalities from~\cite{Triebel1978}.

\begin{lemma}\label{lemma:interpolation} Let $\eta=1/2-\eps$ and $\beta=1/2+\eps$. For $\eps>0$ small enough, there exist $c>0$ and $\gamma,\theta\in (0,1)$ such that
		\begin{equation*}
			\|f\|_{\mathcal C^{\gamma}}\le c\|f\|_{\mathcal C^{\eta}}^{\theta}\|f\|_{H^{-\beta}}^{1-\theta}\,\quad\forall f\in\mathcal C^{\eta}\cap H^{-\beta}\,.
		\end{equation*}
\end{lemma}

\begin{proof}[Proof of~\cref{prop:tightness}]
	\emph{Control of the increments in time}.
	We prove that for all $T>0$, there exists $p>0$ such that
	\begin{equation}\label{eq:tightness_criterion}
		\lim_{h\downarrow0}\limsup_{\kappa\rightarrow+\infty}\mathbb E\left[\vphantom{\sup_{s,t\le T}}\right.\sup_{\substack{s,t\le T\\|t-s|\le h}}\|u^{\kappa}_t-u^{\kappa}_s\|^p_{\mathcal C([0,1])}\left.\vphantom{\sup_{s,t\le T}}\right]=0\,.
	\end{equation}

	Let $(b^n)$ be the sequence given by~\eqref{eq:approximate_b}. For every $\kappa>0$, $n\ge0$, we denote by $u^{\kappa,n}$ the solution of~\eqref{eq:SHE} in the sense of~\cref{theorem:existence_uniqueness_smooth}, started from $\mathbb Q_{0,0}^{\kappa,n}$ given by~\eqref{eq:approx_reversible_measure}.
	We fix $\kappa>0$, $n\ge0$, and $T>0$. The Lyons-Zheng decomposition \cite{LyonsZheng} writes for all $h\in H$, for all $0\le t\le T$,

	\begin{equation*}
		(h,u^{\kappa,n}_t-u_0^{\kappa,n})_H=\frac12 M_t+\frac12 (N_T-N_t)\,,
	\end{equation*}
	where $M$ (resp.~$N$) is a martingale in the natural filtration of $u^{\kappa,n}$ (resp.~that of $(u^{\kappa,n}_{T-t})_{t\in[0,T]}$).
	Moreover, the quadratic variations are both equal to

	\begin{equation*}
		\langle M\rangle_t=\langle N\rangle_t=t\|h\|_{H}^2\,.
	\end{equation*}

	\noindent By the Burkholder-Davis-Gundy inequality and the Minkowski inequality, there exists $c_p\in(0,\infty)$ such that

	\begin{equation*}
		\mathbb E\left[|(h,u^{\kappa,n}_t-u^{\kappa,n}_s)_H|^p\right]^{1/p}\le c_p\|h\|_H(t-s)^{1/2}.
	\end{equation*}

	\noindent By Fatou's lemma and the convergence of~\cref{prop:stability}, this implies that
	\begin{equation}\label{eq:tightness_u_1}
		\mathbb E\left[|(h,u^{\kappa}_t-u^{\kappa}_s)_H|^p\right]^{1/p}\le\liminf_{n\rightarrow\infty}\mathbb E\left[|(h,u^{\kappa,n}_t-u^{\kappa,n}_s)|^p\right]^{1/p}\le c_p\|h\|_H(t-s)^{1/2}
	\end{equation}
	which implies, by the Minkowski inequality for $p\ge2$,
	\begin{equation}\label{eq:tightness_u_2}
	\begin{split}
		\mathbb E\left[\|u^{\kappa}_t-u^{\kappa}_s\|_{H^{-\beta}}^p\right]^{2/p}&=\mathbb E\left[\left(\sum_{k\ge1}k^{-2\beta}|(e_k,u_t^{\kappa}-u_s^{\kappa})_H|^2\right)^{p/2}\right]^{2/p}\\
		&\le\sum_{k\ge1}k^{-2\beta}\mathbb E[|(e_k,u^\kappa_t-u^\kappa_s)_H|^p]^{2/p}\le C(t-s)
	\end{split}
	\end{equation}
	since $\beta>1/2$. Note that $C$ does not depend on $\kappa > 0$, nor on $s,t$.

	\noindent Let $0\le s\le t$, $p\ge1$ and $\eta<1/2$. By the triangle inequality, the Minkowski inequality, and stationarity at the second line,
	\begin{equation*}
	\begin{split}
		\mathbb E[\|u_t^{\kappa}-u_s^{\kappa}\|_{\mathcal C^{\eta}}^p]^{1/p}&\le\mathbb E[\|u_s^{\kappa}\|_{\mathcal C^{\eta}}^p]^{1/p}+\mathbb E[\|u_t^{\kappa}\|_{\mathcal C^{\eta}}^p]^{1/p}\\
		&\le 2\mathbb E[\|u_0^{\kappa}\|_{\mathcal C^{\eta}}^p]^{1/p}=2\mathbb Q_{0,0}^{\kappa}\left[\|w\|_{\mathcal C^{\eta}}^p\right]^{1/p}
	\end{split}
	\end{equation*}
	which is finite uniformly over all $\kappa > 1$ by~\cref{lemma:tightness}. This leads to
	\begin{equation}\label{eq:tightness_u_3}
		\sup_{\kappa > 1}\sup_{0\le s\le t}\mathbb E\left[\left\|u^{\kappa}_t-u^{\kappa}_s\right\|^p_{\mathcal C^{\eta}}\right]<\infty\,.
	\end{equation}

	\noindent\cref{lemma:interpolation} implies via Hölder's inequality that there exist $\gamma,\theta\in(0,1)$ such that for all $p$ large enough,

	\begin{equation*}
		\mathbb E\left[\left\|u^{\kappa}_t-u^{\kappa}_s\right\|^p_{\mathcal C^{\gamma}}\right]\lesssim(t-s)^{\frac{p(1-\theta)}{2}}\,,
	\end{equation*}
	uniformly over all $\kappa > 1$ and $0\le s\le t\le T$. Applying Kolmogorov's continuity theorem, this implies that for all $\nu\in(0,(1-\theta)/2)$ and all $p\ge 1$, we have
	\begin{equation*}
		\sup_{\kappa>1}\mathbb E\left[\sup_{s\neq t\in[0,T]}\frac{\left\|u^{\kappa}_t-u^{\kappa}_s\right\|^p_{\mathcal C^{\gamma}}}{|t-s|^{\nu p}}\right]<\infty\,.
	\end{equation*}
	This implies that~\eqref{eq:tightness_criterion} is fulfilled.

	\noindent\emph{Tightness of the one-dimensional marginals}.
	The processes are stationary, so it suffices to prove tightness of $(\mathbb Q_{0,0}^{\kappa})_{\kappa>1}$.
	The latter is a direct consequence of~\cref{lemma:tightness}.
\end{proof}

\subsection{Reflection measure}

In this subsection, we introduce tools to deal with the reflection measure of~\eqref{eq:SPDE_NP}.

Let $\mathbb M$ be the set of all Radon measures $m$ (i.e. finite on compact subsets) on $[0,\infty)\times(0,1)$ satisfying $\int_{[0,T]\times(0,1)}x(1-x)\,m(\dd t,\dd x)<\infty$ for every $T>0$.
We endow $\mathbb M$ with the smallest topology that makes
\begin{equation*}
	m\mapsto\int_{[0,\infty)\times(0,1)}x(1-x)\phi(t,x)\,m(\dd t,\dd x)
\end{equation*}
continuous for all $\phi\in\mathcal C_c([0,\infty)\times[0,1],\mathbb R)$. It was proved in~\cite[Lemma 15]{Etheridge2014}, using classical arguments, that $\mathbb M$ is a Polish space
so tightness in $\mathbb M$ is equivalent to relative compactness.

In the sequel, we rely on the following tightness criterion. To facilitate the writing, we let $\mathcal C^\infty_0([0,1])$ denote the set of all functions in $\mathcal C^\infty([0,1])$ that vanish at $x=0$ and $x=1$.

\begin{lemma}\label{Lemma:TightMandphi}
	Let $(\eta^i)_{i\in I}$ be a family of random elements of $\mathbb M$. Assume that for every function $\phi\in \mathcal C^\infty_0([0,1])$, $(X^i)_{i\in I}$ is tight in $\mathcal C([0,\infty),\R)$ where
	$$ X_t^i := \int_{[0,t]\times[0,1]} \phi(x) \, \eta^i(\dd s, \dd x)\;,\quad t\ge 0\;.$$
	Then $(\eta^i)_{i\in I}$ is tight in $\mathbb M$ and any converging subsequence $(\eta^{i_k})_{k\ge 1}$ with limit, say, $\eta$, satisfies the following property: for any $\phi\in \mathcal C^\infty_0([0,1])$,  $(X^{i_k})_{k\ge 1}$ converges in law to the continuous process
	$$ X_t := \int_{[0,t]\times[0,1]} \phi(x) \, \eta(\dd s, \dd x)\;,\quad t\ge 0\;.$$
\end{lemma}
Note that one (implicit) requirement of the lemma is that each $X^i$ be a continuous process, which is not at all granted for a generic element $\eta^i \in \mathbb M$.
\begin{proof}
	To prove that $(\eta^i)_{i\in I}$ is tight, it suffices to prove that for every $\psi \in \mathcal C_c^\infty([0,\infty)\times[0,1],\mathbb R)$, the family
	$$ \int_{[0,\infty)\times [0,1]} \psi(s,x) \, \eta^i(\dd s,\dd x)\;,\quad i\in I\;,$$
	is tight in $\R$. Given such a function $\psi$, one can find a non-negative $\phi \in \mathcal C^\infty_0([0,1])$ such that $|\psi(t,x) | \le \phi(x)$ for all $t\ge 0 $ and $x\in[0,1]$, so the assertion of the statement ensures tightness of $(\eta^i)_{i\in I}$.
	
	For simplicity, we denote by $(\eta^n)_{n\ge 0}$ a converging subsequence and we let $\eta$ be its limit. Fix $\phi$ as in the statement. Consider a further subsequence $(\eta^{n_k},X^{n_k})_k$ that converges in law to, say, $(\eta,X)$. If we show that
	$$ X_t = \int_{[0,t]\times[0,1]} \phi(x) \, \eta(\dd s, \dd x)\;,\quad t\ge 0\;,$$
	then we will deduce that all converging subsequences of $(X^n)_n$ have the same limit, and therefore that $(\eta_n,X_n)$ converges to $(\eta,X)$.
	
	We use the Skorokhod representation theorem and suppose that the convergence in law of $(\eta^{n_k},X^{n_k})_k$ holds a.s.~on some other probability space. Let us suppose that $\phi\ge0$, the general case is treated by writing $\phi=\phi^+-\phi^-$, where $\phi^+$ and $\phi^-$ both remain in $\mathcal C^\infty_0([0,1])$.
	For every $t>0$, we set
	\begin{equation}\label{eq:Xtilde}
			\widetilde{X}_t=\int_{[0,t]\times[0,1]}\phi(x)\,\eta(\dd s,\dd x)\,.
	\end{equation}
	We aim at proving that the process $\widetilde{X}$ is a.s. continuous and that, as processes, $\widetilde{X}=X$.
		
		\noindent \emph{Proof that $\widetilde{X}$ is continuous}. For $a>0,\eps>0$, we let $\chi_{a,\eps}:\R\to [0,1]$ be any continuous function such that $\1_{[a,a+\eps]}\le\chi_{a,a+\eps}\le\1_{[a-\eps,a+2\eps]}$.
		Let $\delta>0$, $t>0$,
		\begin{equation*}
			|\widetilde{X}_{t+\delta}-\widetilde{X}_{t}|\le\int_{[0,\infty)\times(0,1)}\1_{[t,t+\delta]}(s)\phi(x)\,\eta(\dd s,\dd x)\le\int_{[0,\infty)\times(0,1)}\chi_{t,t+\delta}(s)\phi(x)\,\eta(\dd s,\dd x)\,.
		\end{equation*}
		Because $\chi_{t,t+\delta}\otimes\phi\in\mathcal C_c([0,\infty)\times(0,1))$, it follows by the definition of the topology on $\mathbb M$ that one has the a.s. limit
		\begin{equation*}
			\int_{[0,\infty)\times(0,1)}\chi_{t,t+\delta}(s)\phi(x)\,\eta(\dd s,\dd x)=\lim_{k\rightarrow\infty}\int_{[0,\infty)\times(0,1)}\chi_{t,t+\delta}(s)\phi(x)\,\eta^{n_k}(\dd s,\dd x)\,.
		\end{equation*}
		Furthermore, for every $k\ge0$,
		\begin{equation*}
			\begin{split}
				\int_{[0,\infty)\times(0,1)}\chi_{t,t+\delta}(s)\phi(x)\,\eta^{n_k}(\dd s,\dd x)&\le\int_{[0,\infty)\times(0,1)}\1_{t-\delta,t+2\delta}(s)\phi(x)\,\eta^{n_k}(\dd s,\dd x)\\
				&=X^{n_k}_{t+2\delta}-X^{n_k}_{t-\delta}\,.
			\end{split}
		\end{equation*}
		Also, by continuity of the map $\mathcal C([0,\infty))\ni x\mapsto x_{t+2\delta}-x_{t-\delta}$,
		\begin{equation*}
			\lim_{k\rightarrow\infty}\,X^{n_k}_{t+2\delta}-X^{n_k}_{t-\delta}=X_{t+2\delta}-X_{t-\delta}\,.
		\end{equation*}
		Finally,
		\begin{equation*}
			\limsup_{\delta\downarrow0}|\widetilde{X}_{t+\delta}-\widetilde{X}_{t}|\le \limsup_{\delta\downarrow0}|X_{t+2\delta}-X_{t-\delta}|=0\,.
		\end{equation*}
		This proves right continuity of $\widetilde{X}$, left continuity is proved similarly.
		
		\noindent\emph{Proof that $X=\widetilde X$}. For $a>0,\eps>0$, we let $\lambda_{a,\eps}:\R\to[0,1]$ be any continuous function such that $\1_{[0,a]}\le\lambda_{a,\eps}\le\1_{[0,a+\eps]}$.
		For any $t>0$
		\begin{equation*}
			|X_t-\widetilde X_t|=\lim_{k\rightarrow\infty}|X^{n_k}_t-\widetilde X_t|\,.
		\end{equation*}
		For $n\ge0$, for every $\eps>0$,
		\begin{equation*}
			\begin{split}
				|X^{n_k}_t-\widetilde X_t|=&|\eta^{n_k}(\1_{[0,t]}\otimes\phi)-\eta(\1_{[0,t]}\otimes\phi)|\\
				\le&|\eta^{n_k}(\1_{[0,t]}\otimes\phi)-\eta^{n_k}(\lambda_{t,\eps}\otimes\phi)|+|\eta^{n_k}(\lambda_{t,\eps}\otimes\phi)-\eta(\lambda_{t,\eps}\otimes\phi)|\\
				&+|\eta(\lambda_{t,\eps}\otimes\phi)-\eta(\1_{[0,t]}\otimes\phi)|\,.
			\end{split}
		\end{equation*}
		First,
		\begin{equation*}
			\begin{split}
				|\eta^{n_k}(\1_{[0,t]}\otimes\phi)-\eta^{n_k}(\lambda_{t,\eps}\otimes\phi)|&=\eta^{n_k}(\lambda_{t,\eps}\otimes\phi)-\eta^{n_k}(\1_{[0,t]}\otimes\phi)\\
				&\le\eta^{n_k}(\1_{[0,t+\eps]}\otimes\phi)-\eta^{n_k}(\1_{[0,t]}\otimes\phi)=X_{t+\eps}^{n_k}-X_t^{n_k}\,.
			\end{split}
		\end{equation*}
		Similarly,
		\begin{equation*}
			|\eta(\1_{[0,t]}\otimes\phi)-\eta(\lambda_{t,\eps}\otimes\phi)|\le\widetilde X_{t+\eps}-\widetilde X_t\,.
		\end{equation*}
		Also, because $\lambda_{t,\eps}\otimes\phi\in\mathcal C_c([0,\infty)\times(0,1))$,
		\begin{equation*}
			\lim_{k\rightarrow\infty}|\eta^{n_k}(\lambda_{t,\eps}\otimes\phi)-\eta(\lambda_{t,\eps}\otimes\phi)|=0
		\end{equation*}
		As a consequence
		\begin{equation*}
			\lim_{k\rightarrow\infty}|X^{n_k}_t-\widetilde X_t|\le|X_{t+\eps}-X_t|+|\widetilde X_{t+\eps}-\widetilde X_t|\,.
		\end{equation*}
		Letting $\eps\downarrow0$ in the previous equation yields
		\begin{equation*}
			\lim_{k\rightarrow\infty}|X^{n_k}_t-\widetilde X_t|=0\,,
		\end{equation*}
		the desired result. As $\widetilde X$ is a.s. a continuous process, we deduce that, a.s., for every $t>0$, $X_t=\widetilde X_t$.
\end{proof}

\begin{lemma}\label{lemma:continuity_u_eta}
	Let $\psi$ a nonnegative function in $\mathcal C_c([0,\infty)\times(0,1))$, we introduce $G:\mathcal C_c([0,\infty)\times[0,1],\mathbb R)\times\mathbb M\longrightarrow\R$,
	\begin{equation*}
		G(v,\eta):=\int_{[0,\infty)\times(0,1)}v\,\psi\,\dd\eta\,.
	\end{equation*}
	Then the map $G$ is continuous.
\end{lemma}

\begin{proof}
	Let $(v^n,\eta^n)$ be a sequence that converges to $(v,\eta)$ in $\mathcal C([0,\infty)\times[0,1],\mathbb R)\times\mathbb M$. We bound
	\begin{equation*}
		\left|G(v^n,\eta^n)-G(v,\eta)\right|
	\end{equation*}
	by 
	\begin{equation*}
		|G(v^n,\eta^n)-G(v,\eta^n)|+\left|G(v,\eta^n)-G(v,\eta)\right|\,.
	\end{equation*}
	Let $T$ be such that $\psi$ has its support in $[0,T]\times(0,1)$. The first term is bounded by 
	\begin{equation*}
		\sup_{t\le T}\|v_t-v^n_t\|_{\infty}\int_{[0,\infty)\times(0,1)}\psi\,\dd\eta^n
	\end{equation*}
	and thus goes to $0$ as $n\rightarrow\infty$ because $v^n$ goes to $v$ uniformly and $\eta^n(\psi)$ goes to $\eta(\psi)$ by vague convergence.
	The second term goes to $0$ because $v\psi\in\mathcal C_c([0,\infty)\times(0,1))$, by vague convergence.
\end{proof}

\subsection{Identification of the limit}\label{Subsec:IdLimit}

In~\cref{prop:tightness} we proved tightness of $(u^\kappa)_{\kappa > 1}$. It remains to show that the limit of any converging subsequence is a stationary solution of the Reflected Stochastic Heat Equation~\eqref{eq:SPDE_NP}. At this point, the main difficulty consists in identifying the random reflection measure. We will proceed in a slightly indirect way.

Let $(b^n)_n$ be the sequence of functions given by~\eqref{eq:approximate_b}. Let $u^{\kappa,n}$ be the solution to~\eqref{eq:SHE} (in the sense of~\cref{theorem:existence_uniqueness_smooth}) with drift $b^n$ started from $u^{\kappa,n}_0\eqlaw\mathbb Q_{0,0}^{\kappa,n}$. Informally, we will identify a random measure $\eta^{\kappa,n}$ in the dynamics of $u^{\kappa,n}$ that will converge as $n\to\infty$ and then $\kappa\to+\infty$, to the reflection measure of~\eqref{eq:SPDE_NP}.

\begin{proposition}\label{Prop:MildLimitPoint}
	The collection $(u^{\kappa,n},\eta^{\kappa,n}, W)_n$ is tight in
	$$ \mathcal C([0,\infty)\times [0,1]) \times \mathbb M \times \mathcal C([0,\infty), \mathcal S'([0,1]))$$
	and any limit point $(\tilde{u}^\kappa,\tilde{\eta}^\kappa,W)$ satisfies:\begin{enumerate}
		\item $\tilde{u}^\kappa$ has the same law as $u^\kappa$, and $\tilde{u}^\kappa(0,\cdot)$ is independent from $W$,
		\item For any given $\phi\in \mathcal C^\infty_0([0,1])$, a.s.~for all $t>0$
		\begin{equation}\label{Eq:tildeukappa}	(\tilde{u}^{\kappa}_t,\phi)-(\tilde{u}^{\kappa}_0,\phi)-\int_0^t(\tilde{u}^{\kappa}_s,\frac12\phi'')\,\dd s-\int_{[0,t]\times[0,1]}\phi(x)\,\tilde{\eta}^{\kappa}(\dd s,\dd x)-W_t(\phi) = 0\,.
		\end{equation}
		\item For any given $\psi \in \mathcal C_c([0,\infty)\times(0,1))$, a.s.~$\int \psi(t,x) |\tilde{u}^\kappa(t,x)| \, \dd \tilde{\eta}^\kappa(\dd t,\dd x) = 0$.
	\end{enumerate}
\end{proposition}
\begin{remark}
	Let us point out that we do not have at our disposal an analytically weak characterisation of the solution of~\eqref{eq:SPDE_skew}. With such a characterisation, and certainly some additional work, one could identify uniquely the law of all limit points $(\tilde{u}^\kappa,\tilde{\eta}^\kappa,W)$. Nevertheless, we will be able to show that as $\kappa\uparrow \infty$, we recover a unique limit in law, whatever choice of limits points is made for each $\kappa>0$: this will be the content of the proof of~\cref{theorem:dynamical_convergence} presented below.
\end{remark}
\begin{proof}
	Recall that $u^{\kappa,n}$ is a stationary solution of~\eqref{eq:SHE} with $b^n$ given by~\eqref{eq:approximate_b}. We define a random measure on $[0,\infty)\times(0,1)$
	\[
	\eta^{\kappa,n}(\dd t,\dd x):=b^n(u_t^{\kappa,n})\,\dd t\,\dd x\,.
	\]
	Since $b^n$ is a bounded function, this measure belongs to $\mathbb M$. \cref{prop:stability} implies that the family $(u^{\kappa,n})_n$ is tight. In addition, $W$ is a space-time white noise, so its law does not depend on $\kappa$ and $n$. The weak formulation~\eqref{eq:she_weak} of the SPDE satisfied by $u^{\kappa,n}$ yields for any $\phi\in\mathcal C^\infty_0([0,1])$
	\begin{equation}\label{Eq:tildeukappan}
		(\tilde{u}^{\kappa,n}_t,\phi)-(\tilde{u}^{\kappa,n}_0,\phi)-\int_0^t(\tilde{u}^{\kappa,n}_s,\frac12\phi'')\,\dd s-X^{\kappa,n}_t-W_t(\phi) = 0\;,
	\end{equation}
	where
	$$ X^{\kappa,n} := \int_0^t \phi(x) \, \eta^{\kappa,n}(\dd s, \dd x)\;,\quad t\ge 0\,, n\ge 1\;.$$
	From this identity and the previous arguments, we deduce that $(X^{\kappa,n})_n$ is tight in $\mathcal C([0,\infty),\R)$. In turn, by~\cref{Lemma:TightMandphi} we deduce that $(\eta^{\kappa,n})_{n\ge 1}$ is tight in $\mathbb M$.
	
	As the product of compact sets is compact in the product topology, $(u^{\kappa,n},\eta^{\kappa,n},W)_n$ is tight in the product topology. Let $(\tilde{u}^{\kappa},\tilde{\eta}^{\kappa},\tilde{W})$ be the limit of a converging subsequence, that we still denote $(u^{\kappa,n},\eta^{\kappa,n},W)_n$ for convenience.
	
	Let us now check that each of the three assertions of the statement holds true. The first assertion is a direct consequence of~\cref{prop:stability}. Regarding the second assertion, using~\cref{Lemma:TightMandphi} we can pass to the limit on~\eqref{Eq:tildeukappan} and obtain~\eqref{Eq:tildeukappa}.
	
	Let us finally prove the third assertion. For convenience, by Skorokhod representation theorem we can assume that $(u^{\kappa,n},\eta^{\kappa,n},W^n)_n$ converges almost surely to $(\tilde{u}^{\kappa},\tilde{\eta}^{\kappa},\tilde{W})$.	We use~\cref{lemma:continuity_u_eta} to write that for any $\psi\in\mathcal C_c([0,\infty)\times(0,1),\mathbb R_+)$,
	\begin{equation*}
		\int_{[0,\infty)\times(0,1)}\tilde{u}^{\kappa}\psi\,\dd\tilde{\eta}^{\kappa}=\lim_n\int_{[0,\infty)\times(0,1)}u^{\kappa,n}\psi\,\dd\eta^{\kappa,n}
	\end{equation*}
	Moreover, for any $A>0$, for every $n\ge0$,
	\begin{align*}
		&\int_{[0,\infty)\times(0,1)}|u^{\kappa,n}|\,\psi\,\dd\eta^{\kappa,n}\\
		= &\int_{[0,\infty)\times(0,1)}\1_{(|u^{\kappa,n}|\le A)}|u^{\kappa,n}|\,\psi\,\dd\eta^{\kappa,n}+\int_{[0,\infty)\times(0,1)}\1_{(|u^{\kappa,n}|>A)}|u^{\kappa,n}|\,\psi\,\dd\eta^{\kappa,n}\\
		\le&A\eta^{\kappa,n}(\psi)+\kappa g_{1/n}(A)\int_{[0,\infty)\times(0,1)}|u^{\kappa,n}(t,x)|\,\psi(t,x)\,\dd t\,\dd x\,.
	\end{align*}
	Notice that $\eta^{\kappa,n}(\psi)$ converges to $\tilde{\eta}^\kappa(\psi)$ as $n\rightarrow\infty$, a quantity which is a.s. finite.
	Also, $g_{1/n}(A)$ goes to $0$ as $n\rightarrow\infty$, and by the Dominated Convergence theorem ($u^{\kappa,n}$ goes to $\tilde{u}^\kappa$ uniformly on compact sets, and $\psi$ has compact support),
	\begin{equation*}
		\lim_{n}\int_{[0,\infty)\times(0,1)}|u^{\kappa,n}(t,x)|\,\psi(t,x)\,\dd t\,\dd x=\int_{[0,\infty)\times(0,1)}|\tilde{u}^\kappa(t,x)|\,\psi(t,x)\,\dd t\,\dd x\,,
	\end{equation*}
	a quantity which is a.s. finite. Combining the previous convergences,
	\begin{equation*}
		\lim_n\int_{[0,\infty)\times(0,1)}|u^{\kappa,n}|\,\psi\,\dd\eta^{\kappa,n} \le A\tilde{\eta}^\kappa(\psi)\,.
	\end{equation*}
	As this is true for every $A>0$, we let $A\downarrow0$ to conclude that $\int_{[0,\infty)\times(0,1)}|\tilde{u}^{\kappa}|\,\psi\,\dd\tilde{\eta}^{\kappa}=0$ a.s.
	Skorokhod representation theorem was harmless because being equal to $0$ a.s.~implies being equal in law to $0$.
	As a consequence, one has that for all $\psi\in\mathcal C_c([0,\infty)\times(0,1),\mathbb R_+)$,
	\begin{equation}\label{eq:reflection_measure}
		\int_{[0,\infty)\times(0,1)}|\tilde{u}^{\kappa}|\,\psi\,\dd\tilde{\eta}^{\kappa}=0\,.
	\end{equation}
	The proof is complete.
\end{proof}
With this result at hand, we can proceed with the proof of our main result.
\begin{proof}[Proof of~\cref{theorem:dynamical_convergence}]
	By~\cref{prop:tightness}, $(u^\kappa)_{\kappa > 1}$ is tight. It remains to show that the limit of any converging subsequence coincides in law with $u^+$. Let $(u^{\kappa_k})_{k\ge 1}$ be a converging subsequence. Now for every $k\ge 1$, let $(\tilde{u}^{\kappa_k}, \tilde{\eta}^{\kappa_k},W)$ be an arbitrary limit point of a converging subsequence of $(u^{\kappa_k,n},\eta^{\kappa_k,n}, W)_n$, which is tight by~\cref{Prop:MildLimitPoint}. Still by~\cref{Prop:MildLimitPoint}, for every $k\ge 1$ and every $\phi\in \mathcal C^\infty_0([0,1])$, a.s.~for all $t>0$
	\begin{equation}\label{Eq:tildeu}
		(\tilde{u}^{\kappa_k}_t,\phi)-(\tilde{u}^{\kappa_k}_0,\phi)-\int_0^t(\tilde{u}^{\kappa_k}_s,\frac12\phi'')\,\dd s-\int_{[0,t]\times[0,1]}\phi(x)\,\tilde{\eta}^{\kappa_k}(\dd s,\dd x)-W_t(\phi) = 0\;.
	\end{equation}
	Since $\tilde{u}^{\kappa_k}$ coincides in law with $u^{\kappa_k}$, we deduce that $(\tilde{u}^{\kappa_k})_{k\ge 1}$ converges in law, and is therefore tight. Moreover the law of $W$ does not depend on $k\ge1$. By~\eqref{Eq:tildeu} we deduce that the sequence of processes $(\int_{[0,t]\times[0,1]}\phi(x)\,\tilde{\eta}^{\kappa_k}(\dd x,\dd s),t\ge 0)_{k\ge 1}$ is tight in $\mathcal C([0,\infty),\R)$. Since this holds for any given $\phi\in \mathcal C^\infty_0([0,1])$, \cref{Lemma:TightMandphi} ensures that $(\tilde{\eta}^{\kappa_k})_{k\ge 1}$ is tight.
	
	The sequence $(\tilde{u}^{\kappa_k}, \tilde{\eta}^{\kappa_k},W)_{k\ge 1}$ is therefore tight. Up to extracting a further subsequence, we can assume that it converges in law: denote by $(\tilde{u}^\infty,\tilde{\eta}^\infty,W)$ the limit. Note that $\tilde{u}^{\kappa_k}(0,\cdot)$ is independent of $W$ for each $k\ge 1$, so that the same property holds at the limit.\\
	Let us check that $(\tilde{u}^\infty,\tilde{\eta}^\infty,W)$ satisfies all the conditions listed in~\cref{Def:NP}.
	
	\noindent \cref{item:NP_1} We have chosen the sequence of initial conditions $u_0^{\kappa}\overset{\text{law}}{=}\mathbb Q_{0,0}^{\kappa}$ so by~\cref{theorem:static_convergence}, $\tilde{u}^\infty_0\overset{\text{law}}{=}\mathbb P_{0,0}^3$. Moreover $\tilde{u}^{\infty}$ is stationary as the limit of stationary processes, so for every $t>0$, $\tilde{u}^\infty_t\overset{\text{law}}{=}\mathbb P_{0,0}^3$. Since in addition, $\tilde{u}^\infty$ is continuous in time, we deduce that $\tilde{u}^\infty$ is a.s. nonnegative and vanishes at $x=0,1$.
	
	\noindent \cref{item:NP_2} For every $\eps>0$ and $T>0$
	\begin{equation*}
		\tilde{\eta}^{\infty}([0,T]\times(\eps,1-\eps))\le C_\eps \int_{[0,T]\times [0,1]} x(1-x) \tilde{\eta}^\infty(\dd s,\dd x)<\infty\,,
	\end{equation*}
	where $C_\eps > 0$ is a constant only depending on $\eps$, and the fact that the latter quantity is finite follows from $\tilde{\eta}^\infty\in\mathbb M$.
	
	\noindent \cref{item:NP_3} Fix $\phi \in \mathcal C^\infty_0([0,1])$. By~\cref{Lemma:TightMandphi}, we know that $(\tilde{u}^{\kappa_k}, X^k,W)_{k\ge 1}$ converges in law to $(\tilde{u}^\infty,X,W)$ where
	$$ X^k_t := \int_{[0,t]\times[0,1]}\phi(x)\,\tilde{\eta}^{\kappa_k}(\dd s,\dd x)\;,$$
	and
	$$ X_t := \int_{[0,t]\times[0,1]}\phi(x)\,\tilde{\eta}^{\infty}(\dd s,\dd x)\;.$$
	Passing to the limit in~\eqref{Eq:tildeu}, we deduce that a.s.~for all $t>0$
	\begin{equation}\label{Eq:tildeulimit}
		(\tilde{u}^{\infty}_t,\phi)-(\tilde{u}^{\infty}_0,\phi)-\int_0^t(\tilde{u}^{\infty}_s,\frac12\phi'')\,\dd s-\int_{[0,t]\times[0,1]}\phi(x)\,\tilde{\eta}^{\infty}(\dd s,\dd x)-W_t(\phi) = 0\;.
	\end{equation}
	\noindent \cref{item:NP_4} Fix $\psi \in\mathcal C_c([0,\infty)\times[0,1])$. By~\cref{Prop:MildLimitPoint} a.s.~for all $k\ge 1$
	\[
	\int_{[0,\infty)\times(0,1)}\psi\, |\tilde{u}^{\kappa_k}|\,\dd\tilde{\eta}^{\kappa_k}=0\,.
	\]
	By~\cref{lemma:continuity_u_eta}
	\[
	\int_{[0,\infty)\times(0,1)}\psi\, |\tilde{u}^{\infty}|\,\dd\tilde{\eta}^{\infty}=0\,.
	\]
	We claim that this implies that 
	\begin{equation}\label{eq:reflection_measure_infty}
		\int_{[0,\infty)\times(0,1)}\tilde{u}^{\infty}\,\dd\tilde{\eta}^{\infty}=0\,.
	\end{equation}
	Indeed, we let $(a_j,b_j,T_j)_j$ be three sequences such that $\lim_j a_j=0$, $\lim_j b_j=1$, $\lim_j T_j=\infty$, and we let $(\psi_j)_{j\ge0}$ be a sequence of functions
	in $\mathcal C_c([0,\infty)\times[0,1],\mathbb R)$ such that for every $j$, $\psi_j$ has support in $[0,T_j]\times[a_j,b_j]$ and such that $\psi_j\longrightarrow1$ pointwise a.e. on $[0,\infty)\times[0,1]$.
	Fatou's lemma implies that 
	\begin{equation*}
		\int_{[0,\infty)\times(0,1)}|\tilde{u}^{\infty}|\,\dd\tilde{\eta}^{\infty}\le\liminf_{j\rightarrow\infty}\int_{[0,\infty)\times(0,1)}|\tilde{u}^{\infty}|\,\psi_j\,\dd\tilde{\eta}^{\infty}=0\quad\text{a.s.}
	\end{equation*}
	We have thus proven that $(\tilde{u}^\infty,\tilde{\eta}^\infty)$ is a solution of~\eqref{eq:SPDE_NP}. In particular, this identifies uniquely all possible limits in law of $(u^{\kappa})_{\kappa > 1}$.
\end{proof}

\section*{Acknowledgements} TL wishes to thank Pierre Faugère for several discussions, especially about the proof of~\cref{Lemma:TightMandphi},
as well as Elias Nohra for discussions about bridges of stochastic processes. The work of CL and LZ was partially funded by the ANR project Smooth ANR-22-CE40-0017 and the Institut Universitaire de France. The work of CL was also partially funded by the ANR project RANDOP ANR-24-CE40-3377.

\bibliographystyle{siam}
\bibliography{sample}

\end{document}